\documentclass[leqno,11pt]{amsart}

\usepackage{geometry}

\usepackage{multirow,booktabs}

\usepackage{scalerel}

\usepackage{amsmath, amssymb, amsfonts, latexsym, mdwlist, amsthm}
\usepackage{subfig}
\usepackage{graphicx}
\usepackage{wrapfig}

\usepackage{mathtools}
\usepackage{tikz}
\usetikzlibrary{calc,trees,positioning,arrows,chains,shapes.geometric,%
	decorations.pathreplacing,decorations.pathmorphing,shapes,%
	matrix,shapes.symbols}

\tikzset{
	>=stealth',
	punktchain/.style={
		rectangle,
		rounded corners,
		draw=black, thick,
		minimum height=3em,
		text centered,
		on chain},
	line/.style={draw, thick, <-},
	element/.style={
		tape,
		top color=white,
		bottom color=blue!50!black!60!,
		minimum width=8em,
		draw=blue!40!black!90, very thick,
		text width=10em,
		minimum height=3.5em,
		text centered,
		on chain},
	every join/.style={->, thick,shorten >=1pt},
	decoration={brace},
	tuborg/.style={decorate},
	tubnode/.style={midway, right=2pt},
}

\usepackage[all]{xy} 

\usetikzlibrary{patterns}

\usepackage{geometry}

\usepackage{enumerate}

\usepackage{amsmath, amssymb, amsfonts, latexsym, mdwlist, amsthm}
\usepackage{subfig}
\usepackage{graphicx}
\usepackage{wrapfig}

\usepackage{mathtools}

\usepackage[bookmarks, colorlinks, breaklinks, pdftitle={},
pdfauthor={Soheyla Feyzbakhsh}]{hyperref}
\hypersetup{linkcolor=blue,citecolor=blue,filecolor=black,urlcolor=blue}

\usepackage{tikz}
\usetikzlibrary{calc,trees,positioning,arrows,chains,shapes.geometric,%
    decorations.pathreplacing,decorations.pathmorphing,shapes,%
    matrix,shapes.symbols}

\tikzset{
>=stealth',
  punktchain/.style={
    rectangle,
    rounded corners,
    draw=black, thick,
    minimum height=3em,
    text centered,
    on chain},
  line/.style={draw, thick, <-},
  element/.style={
    tape,
    top color=white,
    bottom color=blue!50!black!60!,
    minimum width=8em,
    draw=blue!40!black!90, very thick,
    text width=10em,
    minimum height=3.5em,
    text centered,
    on chain},
  every join/.style={->, thick,shorten >=1pt},
  decoration={brace},
  tuborg/.style={decorate},
  tubnode/.style={midway, right=2pt},
}

\usepackage{paralist}
\setdefaultenum{(a)}{(i)}{}{}
\usepackage{enumitem} 

\newcommand\aaa{\operatorname{\textbf{(A)}}}

\def\abs#1{\left\lvert#1\right\rvert}

\makeatletter
\newtheorem*{rep@theorem}{\rep@title}
\newcommand{\newreptheorem}[2]{%
\newenvironment{rep#1}[1]{%
 \def\rep@title{#2 \ref{##1}}%
 \begin{rep@theorem}}%
 {\end{rep@theorem}}}
\makeatother

\newtheorem{Thm}{Theorem}[section]
\newreptheorem{Thm}{Theorem}
\newtheorem{Prop}[Thm]{Proposition}

\newtheorem{Lem}[Thm]{Lemma}

\newtheorem{Cor}[Thm]{Corollary}
\newreptheorem{Cor}{Corollary}

\newreptheorem{Con}{Conjecture}

\newtheorem{thm-int}{Theorem}

\theoremstyle{definition}
\newtheorem{Def-s}[Thm]{Definition}
\newtheorem{Def}[Thm]{Definition}
\newtheorem{Rem}[Thm]{Remark}

\newcommand{\ignore}[1]{}

\begin{document}

\title[Mukai's Program]{Mukai's Program (reconstructing a K3 surface \\ from a curve) via wall-crossing}

\author{Soheyla Feyzbakhsh}
\address{School of Mathematics, Imperial College London, Huxley Building, South Kensington Campus, SW7 2AZ, London, United Kingdom}
\email{s.feyzbakhsh@imperial.ac.uk}


\begin{abstract}
	Let $C$ be a curve of genus $g \geq 11$ such that $g-1$ is a composite number. Suppose $C$ is on a K3 surface whose Picard group is generated by the curve class $[C]$ . We use wall-crossing with respect to Bridgeland stability conditions to generalise Mukai's program to this situation: we show how to reconstruct the K3 surface containing the curve $C$ as a Fourier-Mukai transform of a Brill-Noether locus of vector bundles on $C$.
\end{abstract}

\vspace{-1em}
\maketitle

\section{Introduction}\label{sec:intro}
In this paper, we consider the problem of reconstructing a K3 surface from a curve on that surface. The main result is the following which extends a program proposed by Mukai in \cite[Section 10]{mukai:non-abelian-brill-noether}.
\begin{Thm}\label{main.3}
	Let $(X,H)$ be a polarised K3 surface with Pic$(X) = \mathbb{Z}.H$. Let $C$ be any curve in the linear system $\abs{H}$ of genus $g \geq 11$ such that $g-1$ is a composite number \footnote{In the published version of this paper \cite{feyz:mukai-program}, we also considered the case (B) where $g-1$ is a prime number. But there is a mistake in the proof of \cite[Proposition 5.2 (a)]{feyz:mukai-program} which affects the validity of \cite[Theorem 1.2]{feyz:mukai-program} in case (B), see \cite[Remark 4.1]{feyz:mukai-program-ii} for details. 
		That is why this case has been investigated in another paper \cite{feyz:mukai-program-ii} using a new treatment so that eventually the main result (Theorem 1.1) in \cite{feyz:mukai-program} is proved valid. }. Then $X$ is the unique K3 surface of Picard rank one and genus $g$ containing $C$, and can be reconstructed as a Fourier-Mukai partner of a certain Brill-Noether locus of vector bundles on $C$. 
\end{Thm}
Any K3 surface of Picard rank one has a canonical primitive polarisation and therefore a well-defined genus. Note that the curve $C$ in Theorem \ref{main.3} is not necessarily smooth and can be singular. The missing cases where $g -1$ is a prime number are considered in \cite{feyz:mukai-program-ii}.

Write $g= rs+1$ for two integers $r\geq 2$ and $s\geq \max\{r,5\}$. We consider the Brill-Noether locus $ \mathcal{BN}\coloneqq M_C(r,2rs,r+s)$ of slope semistable rank $r$-vector bundles on the curve $C$ having degree $2rs$ and possessing at least $r+s$ linearly independent global sections. Let $M_{X,H}(v)$ be the moduli space of $H$-Gieseker semistable sheaves with Mukai vector $v = (r,H,s)$ on $X$. We have chosen the Mukai vector $v$ such that it is a primitive class with $v^2 =0$, hence $M_{X,H}(v)$ is a K3 surface as well. Moreover, any $H$-Gieseker semistable sheaf $E \in M_{X,H}(v)$ is a slope stable locally free sheaf. The choice of the Brill-Noether locus $\mathcal{BN}$ is justified by the following Theorem. 
\begin{Thm}\label{main.2}
	Let $(X,H)$ be a polarised K3 surface with Pic$(X) = \mathbb{Z}.H$ and let $C$ be any curve in the linear system $|H|$. We have an isomorphism 
	\begin{equation}\label{function}
	\psi \colon M_{X,H}(v) \rightarrow \mathcal{BN}
	\end{equation}   
	with $\mathcal{BN}$ as defined above, which sends a bundle $E$ on $X$ to its restriction $E|_C$.
\end{Thm}    
In other words, special vector bundles on the curve $C$, which have an unexpected number of global sections, are the restriction of vector bundles on the surface $X$. This is analogous to the case of line bundles, where a well-known theorem by Green and Lazarsfeld \cite{lazarsfeld:special-divisor-on-a-k3-surface} says that the Clifford index of a non-Clifford general curve on a K3 surface can be computed by the restriction of a line bundle on the surface. \par 
There exists a Brauer class $\alpha \in Br(\mathcal{BN})$ and a universal $(1 \times \alpha)$-twisted sheaf $\mathcal{E}$ on $C \times \mathcal{BN}$. Define $v' \in H^*\big(\mathcal{BN},\mathbb{Z}\big)$ to be the Mukai vector of $\mathcal{E}|_{p \times \mathcal{BN}}$ for a point $p$ on the curve $C$ (see \cite{huybrechts:equivalence-of-twisted-k3-surfaces} for definition in case $\alpha \neq 1$).    
\begin{Thm}\label{main.1}
	Let $(X,H)$ be a polarised K3 surface with Pic$(X) = \mathbb{Z}.H$ of genus $g \geq 11$ such that $g-1$ is a composite number, and let $C$ be any curve in the linear system $|H|$.  
	Then any K3 surface of Picard rank one and genus $g$ which contains the curve $C$ is isomorphic to the moduli space $M_{\mathcal{BN},H'}^{\alpha}(v')$ of $\alpha$-twisted sheaves on $\mathcal{BN}$ of Mukai vector $v'$ which are semistable with respect to a generic polarisation $H'$ on $\mathcal{BN}$.    
\end{Thm}
The embedding of the curve $C$ into the K3 surface $M_{\mathcal{BN},H'}^{\alpha}(v')$ is given by $p \mapsto \mathcal{E}|_{p \times \mathcal{BN}}$. Combining Theorems \ref{main.2} and \ref{main.1} gives Theorem \ref{main.3}. 
\subsection{Previous work} \label{pr}
Let $\mathcal{F}_g$ be the moduli space of polarised K3 surfaces $(X,H)$ where $H$ is a primitive ample line bundle on $X$ and $H^2 = 2g-2$. This space is a quasi-projective variety of dimension $19$. Let $\mathcal{P}_g$ be the moduli space of triples $(X,H,C)$ such that $(X,H) \in \mathcal{F}_g$ and $C$ is a smooth curve in the linear system $|H|$. Therefore, its dimension is $19+g$. Finally, let $\mathcal{M}_g$ be the moduli space of smooth curves of genus $g$. Its dimension is $3g-3$. The space $\mathcal{P}_g$ has natural projections to $\mathcal{F}_g$ and $\mathcal{M}_g$ which we denote by $\phi_g$ and $m_g$, respectively;
\begin{equation*}\label{stacks}
\xymatrix{
	& \mathcal{P}_g \ar[dr]^-{\phi_g} \ar[dl]_-{m_g} \\
	\mathcal{M}_g&&\mathcal{F}_g }
\end{equation*}
The map $m_g$ is dominant for $g\leq 11$ and $g \neq 10$ \cite{mukai:curve-k3surface-genus-less-than-10}. In \cite[Theorem 5]{cilebreto:projective-degeneration-of-k3-surface}, Ciliberto, Lopez and Miranda proved that for $g \geq 11$ and $g \neq 12$, the map $m_g$ is birational onto its image. For the exceptional cases $g=10$ or $g=12$, the map $m_{g}$ is neither dominant nor generically finite \cite{mukai:non-abelian-brill-noether}. \par 
In \cite{mukai:non-abelian-brill-noether}, Mukai introduced a geometric program to find the rational inverse of $m_g$ where $g=2s+1$ and $s \geq 5$ odd. His idea to reconstruct the K3 surface is as follows. Let $C$ be a general curve in the image of $m_g$. Consider the Brill-Noether locus $M_C(2,K_C,s+2)^{\text{st}}$ of stable rank $2$-vector bundles on the curve $C$ with canonical determinant and possessing at least $s+2$ linearly independent global sections. Then $M_C(2,K_C,s+2)^{\text{st}}$ is a K3 surface and the K3 surface containing the curve $C$ can be obtained uniquely as a Fourier-Mukai transform of the Brill-Noether locus.\par
This program was completely proved by him in \cite{mukai-curves-k3surface-genus-11} for $g=11$. The key idea is that all vector bundles in the Brill-Noether locus $M_C(2,K_C,7)$ are the restriction of vector bundles on the surface. He first considers a point $(X',C') \in \mathcal{P}_g$ of a special type and shows that the Brill-Noether locus $M_{C'}(2,K_{C'},7)$ is isomorphic to $X'$. Indeed, he proves that both surfaces are isomorphic to the moduli space $M_{X',H'}(v)$ where $v = (2,H,5)$. Given a general pair $(X,C) \in \mathcal{P}_g$, the Brill-Noether locus $M_C(2,K_C,7)$ is a flat deformation of $M_{C'}(2,K_{C'},7)$ and has expected dimension. Thus, it is again a K3 surface and the original K3 surface can be obtained as an appropriate Fourier-Mukai transform of it. \par
Arbarello, Bruno and Sernesi \cite{arbarello:maukai-program} generalised this strategy to higher genera. They proved that for a general pair $(X,C) \in \mathcal{P}_g$ where $g=2s+1 \geq 11$, there is a unique irreducible component $V_C$ of $M_C(2,K_C,s+2)$ such that $(V_C)_{\text{red}}$ is a K3 surface isomorphic to the moduli space $M_{X,H}(v)$ where $v = (2,H,s)$. Then they showed that the original K3 surface can be reconstructed using this component whenever $g \equiv 3$ mod $4$.\par
In this paper, without any deformation argument, we show that for a general pair $(X,C) \in \mathcal{P}_g$, when $g=rs+1 \geq 11$, the Brill-Noether locus $M_C(r,K_C,s+r)$ is isomorphic to the moduli space $M_{X,H}(r,H,s)$. As a result, we prove the uniqueness of the K3 surface of Picard rank one which contains the curve $C$ of genus $g \geq 11$ when $g-1$ is a composite number.
\subsection{The strategy of the proof}
We prove Theorem \ref{main.2} by wall-crossing for the push-forward of semistable vector bundles on the curve $C$, with respect to Bridgeland stability conditions on the bounded derived category $\mathcal{D}(X)$ of $X$. There exists a region in the space of stability conditions where the Brill-Noether behaviour  of \emph{stable} objects is completely controlled by the nearby \emph{Brill-Noether wall}. This wall destabilises objects with non-zero global sections, and arguments similar to \cite{bayer:brill-noether} show that the Brill-Noether loci are mostly of expected dimension. Our first key result, Proposition \ref{polygon}, gives an extension to \emph{unstable} objects: it gives a bound on the number of global sections in terms of their \emph{mass}, i.e. the length of their Harder-Narasimhan polygon.   \par
Consequently, we only need a polygon that circumscribes this Harder-Narasimhan polygon on the left, to bound the number of global sections. For any coherent sheaf, there exists a chamber which is called the Gieseker chamber, where the notion of Bridgeland stability coincides with the old notion of Gieseker stability. Unlike the case of push-forward of line bundles considered in \cite{bayer:brill-noether}, the Brill-Noether wall is not adjacent to the Gieseker chamber for the push-forward of semistable vector bundles $F$ of higher ranks on the curve $C$. However, the wall that bounds the Gieseker chamber provides an extremal polygon which contains the Harder-Narasimhan polygon, see e.g. Lemma \ref{inside poly}. Combined with Proposition \ref{polygon}, this gives a bound on the number of global sections of vector bundles on the curve $C$; the proof also shows that the bound is achieved if and only if the vector bundle $F$ is the restriction of a vector bundle on the surface.   
\subsection*{Plan of the paper} Section \ref{section.2} reviews the definition of geometric stability conditions on K3 surfaces and describes a two-dimension family of stability conditions. Section \ref{section.3} deals with the Brill-Noether wall; we provide an upper bound for the number of global sections via the geometry of Harder-Narasimhan polygon. Section \ref{section.4} concerns the proof of bijectivity of the morphism $\psi$ in \eqref{function}. The proof of the main result is contained in Section \ref{section.6}.

\subsection*{Acknowledgement}I would like to thank Arend Bayer for many useful discussions. I am grateful for comments by Benjamin Bakker, Gavril Farkas, Chunyi Li, Hsueh-Yung Lin, Richard Thomas, Yukinobu Toda and Bach Tran. I would also like to thank the referees for their careful reading of the paper, and for many useful suggestions. The author was supported by the ERC starting grant WallXBirGeom 337039.

\section{Bridgeland stability conditions on K3 surfaces}\label{section.2}
In this section, we give a brief review of a two-dimensional family of Bridgeland stability conditions on the bounded derived category of coherent sheaves on a K3 surface. The main references are \cite{bridgeland:stability-condition-on-triangulated-category,bridgeland:K3-surfaces}.
\subsection{Bridgeland stability conditions} Let $(X,H)$ be a smooth polarised K3 surface with Pic$(X) = \mathbb{Z}.H$. We denote by $\mathcal{D}(X)$ the bounded derived category of coherent sheaves on the surface $X$. The
Mukai vector of an object $E \in \mathcal{D}(X)$ is an element of the lattice $\mathcal{N}(X) = \mathbb{Z} \oplus \text{NS}(X) \oplus \mathbb{Z} \cong \mathbb{Z}^3$ defined via
\begin{equation*}
v(E) = \big(\text{rk}(E),\text{c}(E)H,\text{s}(E)\big) = \text{ch}(E)\sqrt{\text{td}(X)} \in H^*(X,\mathbb{Z}),
\end{equation*}
where ch$(E)$ is the Chern character of $E$. The Mukai bilinear form 
\begin{equation*}
\left \langle v(E), v(E')\right \rangle = \text{c}(E)\text{c}(E')H^2 -\text{rk}(E)\text{s}(E') - \text{rk}(E')\text{s}(E) 
\end{equation*}
makes $\mathcal{N}(X)$ into a lattice of signature $(2,1)$. The Riemann-Roch theorem implies that this form is the negative of the Euler form, defined as
\begin{equation*}
\chi(E,E') = \sum_{i} (-1)^{i} \dim_{\mathbb{C}} \text{Hom}_X^{i}(E,E')  =  -\left \langle v(E), v(E')\right \rangle.
\end{equation*} 
Note that the Euler form $\chi(-,-)$ defines a bilinear form on the Grothendieck group $K(X)$ which descends to a non-degenerate form on the lattice 
\begin{equation*}
\mathcal{N}(X) = K(X)/K(X)^{\perp},
\end{equation*}
where $K(X)^{\perp}$ is the left-radical. 
Recall that for a coherent sheaf $E$ with positive rank $\text{rk}(E) >0$, the slope is defined as 
\begin{equation*}
\mu_H(E) \coloneqq \dfrac{\text{c}(E)}{\text{rk}(E)},
\end{equation*} 
and if $\text{rk}(E) = 0$, define $\mu_H(E) \coloneqq +\infty$.
\begin{Def}
	We say that an object $E \in \text{Coh}(X)$ is $\mu_H$-(semi)stable if for all proper non-trivial subsheaves $F \subset E$, we have $\mu_H(F) < (\leq) \,\mu_H(E)$.	
\end{Def}

A stability function on an abelian category $\mathcal{A}$ is a group homomorphism $Z \colon K(\mathcal{A}) \rightarrow \mathbb{C}$ such that for any non zero object $E \in \mathcal{A}$,
\begin{equation*}
Z(E) \in \mathbb{R}^{>0} \text{exp}(i\pi \phi (E)) \;\; \text{with} \;\; 0 < \phi(E) \leq 1.
\end{equation*}
By \cite[Proposition 3.5]{bridgeland:K3-surfaces}, to give a stability condition on a triangulated category $\mathcal{D}$ is equivalent to giving a bounded t-structure on $\mathcal{D}$ and a stability function on its heart which has the Harder-Narasimhan property. 
Given a real number $b \in \mathbb{R}$, denote by $\mathcal{T}^{b} \subset \text{Coh}(X)$ the subcategory of sheaves $E$ whose quotients $E \twoheadrightarrow F$ satisfy $\mu_H(F) > b$ and by $\mathcal{F}^{b} \subset \text{Coh}(X)$ the subcategory of sheaves $E'$ whose subsheaves $F' \hookrightarrow E'$ satisfy $\mu_H(F') \leq b$. Tilting with respect to the torsion pair $(\mathcal{T}^{b},\mathcal{F}^{b})$ on Coh$(X)$ gives a bounded $t$-structure on $\mathcal{D}(X)$ with heart
\begin{equation*}
\mathcal{A}(b) \coloneqq \{E \in \mathcal{D}(X) \colon E \cong [E^{-1}  \xrightarrow{d} E^0] , \; \text{ker } d \in \mathcal{F}^{b} \; \text{and}\; \text{cok } d \in \mathcal{T}^{b} \} \subset \mathcal{D}(X). 
\end{equation*} 
All the stability functions that we will consider in this paper factor through the surjection
$K(X) \twoheadrightarrow \mathcal{N}(X)$. For a pair $(b,w) \in \mathbb{H} = \mathbb{R} \times \mathbb{R}^{>0}$, the stability function $Z_{(b,w)} \colon \mathcal{N}(X) \rightarrow \mathbb{C}$ is defined as
\begin{equation*}
Z_{(b,w)}(r,cH,s) = \bigg\langle (r,cH,s), \bigg(1,bH, \frac{H^2}{2} (b^2 -w^2) \bigg) \bigg\rangle + i \bigg\langle (r,cH,s),\bigg(0,\frac{H}{H^2},b\bigg) \bigg\rangle.
\end{equation*}
We denote the root system by $\Delta(X) \coloneqq \{ \delta \in \mathcal{N}(X) \colon \; \langle \delta , \delta \rangle =-2 \}$.
\begin{Thm}[\cite{bridgeland:K3-surfaces}]  \label{Bridgeland}
	Suppose $(X,H)$ is a polarised K3 surface with Pic$(X) = \mathbb{Z}.H$. Then the pair $\sigma_{(b,w)} = \big(\mathcal{A}(b),Z_{(b,w)}\big)$ defines a Bridgeland stability condition on $\mathcal{D}(X)$ if for all $\delta \in \Delta(X)$ with rk$(\delta) >0$ and Im$[Z_{(b,w)}(\delta)] =0$ we have Re$[Z_{(b,w)}(\delta)]>0$. 
	The family of stability conditions $\sigma_{(b,w)}$ varies continuously as the pair $(b,w)$ varies in $\mathbb{H}$.
\end{Thm}
Note that the stability condition $\sigma_{(b,w)}$, up to the action of $\tilde{\text{GL}}^{+}(2,\mathbb{R})$, is the same as the stability condition defined in \cite[Section 6]{bridgeland:K3-surfaces}. We expand upon the statements in Theorem \ref{Bridgeland} by explaining the notion of $\sigma_{(b,w)}$-stability and the associated Harder-Narasimhan filtration. For a stability condition $\sigma_{(b,w)}$ and $E \in \mathcal{A}(b)$, we have $Z_{(b,w)}(v(E)) \in \mathbb{R}^{>0}\exp\big(i\pi \phi_{(b,w)}(v(E)) \big)$ where 
\begin{equation*}
\phi_{(b,w)}(v(E)) = \dfrac{1}{\pi}\tan^{-1}\bigg(-\dfrac{\text{Re}[Z_{(b,w)}(v(E))]}{ \text{Im}[Z_{(b,w)}(v(E))]  }  \bigg) + \dfrac{1}{2} \in (0,1].
\end{equation*}
We will abuse notations and write $Z(E)$ and $\phi(E)$ instead of $Z(v(E))$ and $\phi(v(E))$.
\begin{Def}
	We say that an object $E \in \mathcal{D}(X)$ is $\sigma_{(b,w)}$-(semi)stable if some shift $E[k]$ is contained in the abelian category $\mathcal{A}(b)$ and for any non-trivial subobject $E' \subset E[k]$ in $\mathcal{A}(b)$, we have $\phi_{(b,w)}(E') < (\leq) \phi_{(b,w)}(E[k])$.
\end{Def}
Any object $E \in \mathcal{A}(b)$ admits a Harder-Narasimhan (HN) filtration: a sequence 
\begin{equation}\label{filtration}
0=\tilde{E}_0 \subset \tilde{E}_1 \subset \tilde{E}_2 \subset ... \subset \tilde{E}_n=E
\end{equation}
of objects in $\mathcal{A}(b)$ where the factors $E_i \coloneqq \tilde{E}_i/\tilde{E}_{i-1}$ are $\sigma_{(b,w)}$-semistable and 
\begin{equation*}
\phi^{+}_{(b,w)}(E) \coloneqq \phi_{(b,w)}(E_1) > \phi_{(b,w)}(E_2) > ....> \phi_{(b,w)}(E_n) \eqqcolon \phi_{(b,w)}^{-}(E).
\end{equation*} 
In addition, any $\sigma_{(b,w)}$-semistable object $E \in \mathcal{A}(b)$ has a  Jordan-H$\ddot{\text{o}}$lder (JH) filtration into stable factors of the same phase, see \cite[Section 2]{bridgeland:K3-surfaces} for more details. 

Suppose $E_1 \hookrightarrow E_2 \twoheadrightarrow E_3 $ is a short exact sequence in $\mathcal{A}(b)$. Since $H^{i}(E_j) = 0$ for $j=1,2,3$ and $i \neq 0,-1$, taking cohomology gives a long exact sequence of coherent sheaves 
\begin{equation*}
0 \rightarrow H^{-1}(E_1) \rightarrow H^{-1}(E_2) \rightarrow H^{-1}(E_3) \rightarrow H^0(E_1) \rightarrow H^0(E_2) \rightarrow H^0(E_3) \rightarrow 0.
\end{equation*}  
For any pair of objects $E$ and $E'$ of $\mathcal{D}(X)$, Serre duality gives isomorphisms
\begin{equation*}
\text{Hom}_X^i(E,E') \cong \text{Hom}_X^{2-i}(E',E)^*.
\end{equation*} 
If the objects $E$ and $E'$ lie in the heart $\mathcal{A}(b)$, then $\text{Hom}_X^i(E,E') = 0$ if $i<0$ or $i>2$. Suppose the object $E \in \mathcal{A}(b)$ is $\sigma_{(b,w)}$-stable, then $E$ does not have any non-trivial subobject with the same phase, thus $\text{Hom}_X(E,E) = \text{Hom}_X^2(E,E)^* = \mathbb{C}$. This implies 
\begin{equation}\label{bogomolove-k3 surface}
v(E)^2 +2 = \text{Hom}^1_X(E,E) \geq 0.
\end{equation}

To simplify drawing the figures, we always consider the following projection:
\begin{equation*}
pr\colon \mathcal{N}(X) \setminus \{s = 0 \}  \rightarrow \mathbb{R}^2 \;\;,\;\; pr(r,cH,s) = \bigg(\dfrac{c}{s} , \dfrac{r}{s}\bigg).
\end{equation*}  
Take a pair $(b,w) \in \mathbb{H}$, the kernel of $Z_{(b,w)}$ is a line inside the negative cone in $\mathcal{N}(X) \otimes \mathbb{R} \cong \mathbb{R}^3$ spanned by the vector $\big(2,2bH,H^2(b^2+w^2)\big)$. Its projection is denoted by 
\begin{equation*}
k(b,w) \coloneqq pr\big(\ker Z_{(b,w)}\big) = \bigg(\dfrac{2b}{H^2(b^2+w^2)} ,  \dfrac{2}{H^2(b^2+w^2)}    \bigg).
\end{equation*}   
Thus, for any stability condition $\sigma_{(b,w)}$, we associate a point $k(b,w) \in \mathbb{R}^2$. The two dimensional family of stability conditions of form $\sigma_{(b,w)}$, is parametrised by the space 
\begin{equation*}
V(X) \coloneqq \left\{ k(b,w)\colon \;  \text{the pair }\big(\mathcal{A}(b) , Z_{(b,w)}\big) \text{ is a stability condition on } \mathcal{D}(X) \right\}  \subset \mathbb{R}^2
\end{equation*}  
with the standard topology on $\mathbb{R}^2$. 
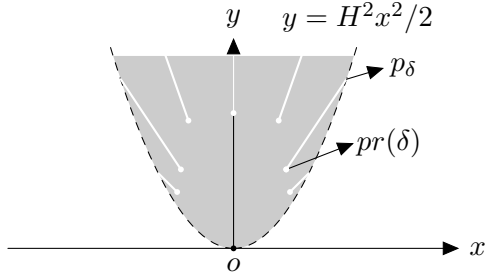
\begin{figure}[h]
	\begin{centering}
		\definecolor{zzttqq}{rgb}{0.27,0.27,0.27}
		\definecolor{qqqqff}{rgb}{0.33,0.33,0.33}
		\definecolor{uququq}{rgb}{0.25,0.25,0.25}
		\definecolor{xdxdff}{rgb}{0.66,0.66,0.66}
		
		\begin{tikzpicture}[line cap=round,line join=round,>=triangle 45,x=1.0cm,y=1.0cm]
		
		\draw[->,color=black] (-3,0) -- (3,0);

		
		

		\fill [fill=gray!40!white] (0,0) parabola (1.6,2.56) parabola [bend at end] (-1.6,2.56) parabola [bend at end] (0,0);
		
		\draw [dashed] (0,0) parabola (1.65,2.72); 
		\draw [dashed] (0,0) parabola (-1.65,2.72);

		\draw [thick, color= white] (0.75,0.75)--(1,1);
		\draw [thick, color=white] (0.7,1.05) -- (1.5,2.25);
		\draw [thick, color=white] (0.6,1.7) -- (0.9148,2.5921);
		
		\draw [thick, color= white] (-0.75,0.75)--(-1,1);
		\draw [thick, color=white] (-0.7,1.05) -- (-1.5,2.25);
		\draw [thick, color=white] (-0.6,1.7) -- (-0.9148,2.5921);

		\draw[color=black] (0,0) -- (0,1.8);
		\draw[->,color=white] (0,1.8) -- (0,2.8);
		\draw[->,color=black] (0,2.7) -- (0,2.8);
		
		
		
		\draw [dashed, color=black] (0.7,1.05) -- (0,0);
		
		
		\draw (1.45,2.25)  node [right ] {$q_{\delta}$};	
		\draw  (0.16,1.1) node [right ] {$p_{\delta}$};
		\draw  (1.65,2.72) node [above] {$y= \frac{H^2x^2}{2}$};
		
		\draw (0,0)  node [below] {$o$};
		
		\draw (0,2.8) node [above] {$y$};
		\draw (3,0) node [right] {$x$};
		
		\begin{scriptsize}
		\fill [color=white] (0.75,0.75) circle (1.1pt);
		\fill [color=white] (0.7,1.05) circle (1.1pt);
		\fill [color=white] (0.6,1.7) circle (1.1pt);  
		
		\fill [color=white] (1,1) circle (1.1pt);
		\fill [color=white] (1.5,2.25) circle (1.1pt);
		\fill [color=white] (-1,1) circle (1.1pt);
		\fill [color=white] (-1.5,2.25) circle (1.1pt);

		
		
		\fill [color=white] (0,1.8) circle (1.1pt);
		\fill [color=white] (0.7,1.05) circle (1.1pt);
		
		\fill [color=white] (-0.75,0.75) circle (1.1pt);
		\fill [color=white] (-0.7,1.05) circle (1.1pt);
		\fill [color=white] (-0.6,1.7) circle (1.1pt);
		\fill [color=black] (0,0) circle (1.1pt);
		
		\end{scriptsize}
		
		\end{tikzpicture}
		
		\caption{The grey area is the 2-dimensional subspace of stability conditions $V(X)$.}
		
		\label{hole}
		
	\end{centering}
\end{figure}	
\begin{Lem}
	We have
	\begin{equation*}
	V(X) = \left\{ (x,y) \in \mathbb{R}^2 \colon \;\; y>\frac{H^2x^2}{2} \right\} \; \setminus \bigcup_{\delta \in \Delta(X)} I_{\delta}
	\end{equation*}
	where $I_{\delta}$ is the closed line segment that connects $p_{\delta} \eqqcolon pr(\delta)$ to $q_{\delta}$ which is 
	the intersection point of the parabola $y=\frac{H^2}{2}x^2$ with the line through the origin and $p_{\delta}$, see Figure \ref{hole}.
	
\end{Lem} 
\begin{proof}
	By definition, the point $k(b,w)$ is above the parabola and for every point $(x,y)$ above the parabola, there exists a unique pair $(b,w) \in \mathbb{H}$ such that $k(b,w)=(x,y)$.
	For any root $\delta \in \Delta(X)$, Im$[Z_{(b,w)}(\delta)] = 0$ if and only if the point $k(b,w)$ is on the line passing through the origin and the point $pr(\delta) = p_{\delta}$. If rk$(\delta)>0$, the line segment $I_{\delta}$ is precisely the segment of this line on which Re$[Z_{(b,w)}(\delta)]>0$, thus the claim follows from Theorem \ref{Bridgeland}. 
\end{proof}
\begin{Rem}\label{plane of having the same phase}
	The point $k(b,w)$ is on the line $x=by$. As $w$ gets larger, the point $k(b,w)$ gets closer to the origin. 
	Take two non-parallel vectors $u,v \in \mathcal{N}(X) \otimes \mathbb{R} \cong \mathbb{R}^3$, then $Z_{(b,w)}(v)$ and $Z_{(b,w)}(u)$ are aligned if and only if the kernel of $Z_{(b,w)}$ in $\mathcal{N}(X) \otimes \mathbb{R}$ lies on the plane spanned by $v$ and $u$, i.e. the points corresponding to $\mathbb{R}.u$, $\mathbb{R}.v$ and $\ker Z_{(b,w)}$ in the projective space $\mathbb{P}_{\mathbb{R}}^2$ are collinear. This, in particular, implies that if three objects $E_1,E_2$ and $E_3$ in $\mathcal{D}(X)$ have the same phase with respect to a stability condition $\sigma_{(b,w)}$, there must be a linear dependence relation among the vectors $v(E_1), v(E_2)$ and $v(E_3)$ in $\mathcal{N}(X) \otimes \mathbb{R}$.
\end{Rem} 
The 2-dimensional family of stability conditions parametrised by the space $V(X)$ admits a chamber decomposition for any object $E \in \mathcal{D}(X)$.
\begin{Prop}\label{line wall}
	Given an object $E \in \mathcal{D}(X)$, there exists a locally finite set of \emph{walls} (line segments) in $V(X)$ with the following properties:
	\begin{itemize*}
		\item[(a)] The $\sigma_{(b,w)}$-(semi)stability or instability of $E$ is independent of the choice of the stability condition $\sigma_{(b,w)}$ in any chamber (which is a connected component of the complement of the union of walls).
		\item[(b)] When $\sigma_{(b_0,w_0)}$ is on a wall $\mathcal{W}_E$, i.e. the point $k(b_0,w_0) \in \mathcal{W}_E$, then $E$ is strictly $\sigma_{(b_0,w_0)}$-semistable.
		\item[(c)] If $E$ is semistable in one of the adjacent chambers to a wall, then it is unstable in the other adjacent chamber.
		\item[(d)] Any wall $\mathcal{W}_E$ is a connected component of $L\cap V(X)$, where $L$ is a line that passes through the point $pr(v(E))$ if $\text{s}(E) \neq 0$, or that has a slope of $\text{rk}(E)/\text{c}(E)$ if $\text{s}(E) = 0$.
	\end{itemize*}
\end{Prop}
\begin{proof}
	The existence of a locally finite set of walls which satisfies properties $(a)$, $(b)$ and $(c)$ is proved in \cite[section 9]{bridgeland:K3-surfaces}, see also \cite{maciocia:the-walls-projective-spaces} for the description of the walls. Remark \ref{plane of having the same phase} implies that a numerical wall for the class $v(E)$ is a line $L$ as claimed in part $(d)$, thus \cite[Proposition 6.22.(7)]{macri:intro-bridgeland-stability} completes the proof of $(d)$. 
\end{proof}
Note that in Proposition \ref{line wall}, we do not assume $v(E)$ is primitive; in particular, $E$ might be strictly semistable in the interior of a chamber.
\begin{Rem}\label{deforming the stability condition along the wall}
	Take an object $E \in \mathcal{D}(X)$, let $L_1$ be a connected component of $L \cap V(X)$ where $L$ is a line as described in Proposition \ref{line wall}, part $(d)$. Suppose $E$ is $\sigma_{(b_0,w_0)}$-(semi)stable for a stability condition $\sigma_{(b_0,w_0)}$ on $L_1$. Then the structure of walls shows that $E$ is (semi)stable with respect to all stability conditions on $L_1$. Moreover, if $E$ is in the heart $\mathcal{A}(b_0)$, by a straightforward computation, one can show that when we deform the stability condition $\sigma_{(b_0,w_0)}$ along the line segment $L_1$, the phase of $E$ is fixed so it remains in the heart.    
\end{Rem}

\subsection*{The two-dimensional subspace of stability conditions } To describe the space $V(X)$, we need to find out the possible positions of the projection of roots. 
We denote by $\gamma_n$ the point $\big(\frac{1}{n} , \frac{H^2}{2n^2}\big)$ on the parabola for any $n \in \mathbb{Q}$, see Figure \ref{no hole.1}.
\begin{Lem}\label{no spherical}
	For any positive number $n  \in \dfrac{1}{2}\mathbb{N}$, define 
	\begin{equation*}
	U_n \coloneqq \left\{ (x,y) \in \mathbb{R}^2 \colon 0 < \abs{x} < \dfrac{1}{n} \; \text{and}\;\; \dfrac{H^2}{2n} \abs{x} < y \right\}.
	\end{equation*}
	If $n\leq \frac{H^2}{2}$, then there is no projection of roots $pr(\delta)$ in $U_n$.
	
	\begin{figure} [h]
		\begin{centering}
			
			\begin{tikzpicture}[line cap=round,line join=round,>=triangle 45,x=1.0cm,y=1.0cm]
			
			\filldraw[fill=gray!40!white, draw=white] (0,0)--(1,1)--(1,2.5)--(-1,2.5)--(-1,1)--(0,0);
			
			\draw[->,color=black] (-3,0) -- (3,0);

			
			

			
			\draw [] (0,0) parabola (1.61,2.5921); 
			\draw [] (0,0) parabola (-1.61,2.5921);

			\draw [ color=black] (0,0)--(1,1);
			\draw [color=black] (0,0)--(-1,1);
			\draw [color=black] (-1,1)--(-1,2.5);
			\draw [color=black] (1,1)--(1,2.5);
			\draw [dashed, color=black] (1,0) -- (1,1);
			\draw [dashed, color=black] (-1,0) -- (-1, 1);
			

			\draw[color=white] (0,1.9) -- (0,2.6);
			\draw [color=black] (0,0)--(0,1.9);
			\draw[->,color=black] (0,2.6) -- (0,2.8);
			
			\draw (-1.1,0) node [below] {$-\frac{1}{n}$};
			\draw (1,0) node [below] {$\frac{1}{n}$};
			\draw  (1.65,2.72) node [above] {$y= \frac{H^2x^2}{2}$};
			\draw (0,2.8) node [above] {$y$};
			\draw (3,0) node [right] {$x$};
			\draw (0,0) node [below] {$o$};
			\draw (-.05,2) node [right] {$o'$};
			\draw (1,1) node [right] {$\gamma_n$};
			\draw (-.93,.9) node [left] {$\gamma_{-n}$};
			
			\begin{scriptsize}
			
			\fill [color=black] (-1,0) circle (1.1pt);
			\fill [color=white] (0,1.9) circle (1.1pt);
			\fill [color=black] (1,0) circle (1.1pt);
			\fill [color=black] (1,1) circle (1.1pt);
			\fill [color=black] (-1,1) circle (1.1pt);
			\fill [color=black] (0,0) circle (1.1pt);
			
			\end{scriptsize}
			
			\end{tikzpicture}
			
			\caption{No projection of roots in the grey area $U_n$}
			
			\label{no hole.1}
			
		\end{centering}
		
	\end{figure}
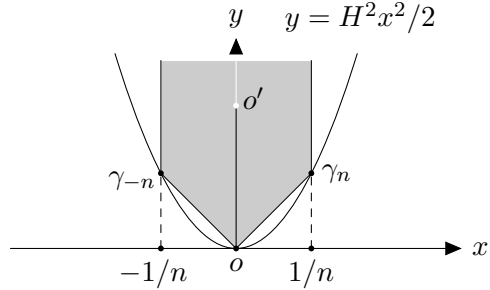
\end{Lem}
\begin{proof}
	Assume for a contradiction that $pr(\delta = (\tilde{r}, \tilde{c}H, \tilde{s})) \in U_n$, then 
	\begin{equation}\label{qq.0}
	0 < \dfrac{H^2}{2n}\abs{\dfrac{\tilde{c}}{\tilde{s}}} < \abs{\dfrac{\tilde{r}}{\tilde{s}}},  
	\end{equation}
	which implies $\abs{\tilde{c}^2H^2} < \abs{2n\tilde{r}\tilde{c}}$. By assumption $\delta^2 = \tilde{c}^2H^2- 2\tilde{r}\tilde{s} =-2$, so
	\begin{equation}\label{qq.1}
	0 < \abs{\tilde{s}- \dfrac{1}{\tilde{r}}} < \abs{n\tilde{c}}.
	\end{equation}	
	Moreover, 
	\begin{equation}\label{qq.2}
	0< \abs{\dfrac{\tilde{c}}{\tilde{s}}} < \dfrac{1}{n} \;\;\; \Rightarrow\;\;\; 0 < \abs{n\tilde{c}} < \abs{\tilde{s}}.
	\end{equation}
	If $n \in \mathbb{N}$, there is no triple $(\tilde{r},\tilde{c},\tilde{s}) \in \mathbb{Z}^3$ that satisfies both inequalities \eqref{qq.1} and \eqref{qq.2} and if $n \in \dfrac{1}{2}\mathbb{N}$, the only possible case is $\tilde{r} = \pm 1$. But we assumed $2n \leq H^2$ and inequality \eqref{qq.0} implies $0 < \abs{\tilde{c}}< 1$, a contradiction. 
\end{proof}
\begin{Rem}\label{y-axis}
	Note that if the point $pr\big(\delta = (\tilde{r},\tilde{c}H,\tilde{s})\big) = (\tilde{c}/\tilde{s} , \tilde{r}/\tilde{s})$ is on the y-axis, then $\tilde{c} = 0$. Since $\delta^2 = -2 \tilde{r}\tilde{s} = -2$, we have $\tilde{r}=\tilde{s} = \pm 1$ and $pr(\delta) = (0,1) = pr(v(\mathcal{O}_X))$. This point is denoted by $o'$ in Figure \ref{no hole.1}. 
\end{Rem}
Given three positive numbers $m, n,\, \epsilon\in \dfrac{1}{2}\mathbb{N}$ such that $m<n$, the point on the line segment $\overline{\gamma_m\gamma_n}$ with the $x$-coordinate $\frac{1}{m+\epsilon}$ is denoted by $q'_{m,n,\epsilon}$. Also, the point where the line segments $\overline{\gamma_m\gamma_n}$ and $\overline{o\gamma_{n-\epsilon}}$ intersect is denoted by $q_{m,n,\epsilon}$, see Figure \ref{no hole.2}. One can define similar points for the triple $(-m,-n,-\epsilon)$. For two points $q_1, q_2 \in \mathbb{R}^2$, we denote by $[\overline{q_1q_2}]$ the closed line segment which contains both $q_1$ and $q_2$. The open line segment which contains neither $q_1$ nor $q_2$ is denoted by $(\overline{q_1q_2})$ and if it contains only $q_1$ not $q_2$ is denoted by $[\overline{q_1q_2})$.  
\begin{Lem}\label{no spherical.2}
	Take $m, \, n,\, \epsilon \in \dfrac{1}{2}\mathbb{N}$ such that $\epsilon + \frac{1}{2} < n \leq \frac{H^2}{2}$ and 
	\begin{equation}\label{assumption.1}
	m < \dfrac{2\epsilon}{2\epsilon+1}n - \epsilon.
	\end{equation}
	Then there is no projection of roots in the grey area in Figure \ref{no hole.2} and on the open line segments $(\overline{q_{m,n,\epsilon}\;q'_{m,n,\epsilon}})$ and $(\overline{q_{-m,-n,-\epsilon}\;q'_{-m,-n,-\epsilon}})$.
	
	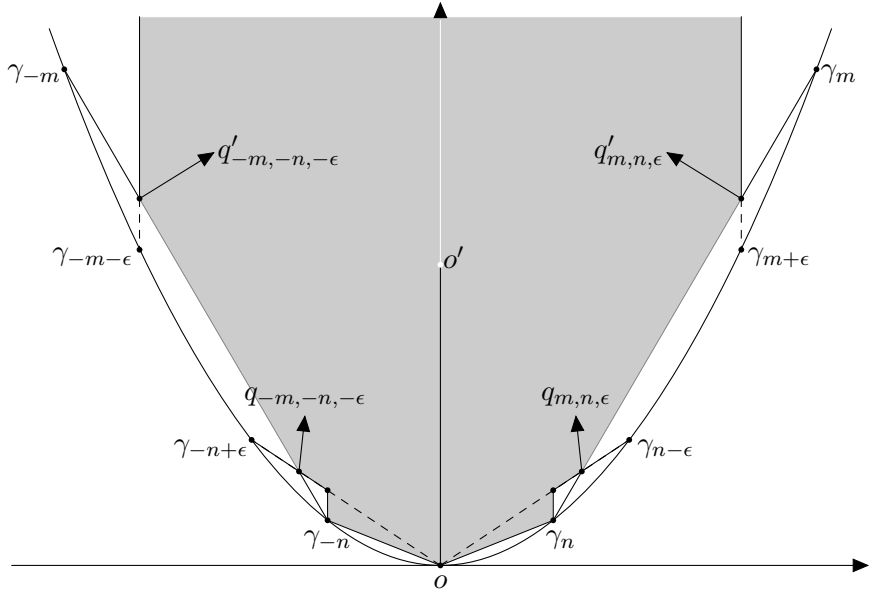
\begin{figure} [h]
		\begin{centering}
			\definecolor{zzttqq}{rgb}{0.27,0.27,0.27}
			\definecolor{qqqqff}{rgb}{0.33,0.33,0.33}
			\definecolor{uququq}{rgb}{0.25,0.25,0.25}
			\definecolor{xdxdff}{rgb}{0.66,0.66,0.66}
			
			\begin{tikzpicture}[line cap=round,line join=round,>=triangle 45,x=1.0cm,y=1.0cm]
			
			\filldraw[fill=gray!40!white, draw=white] (0,0) --(-1.5,3/5)--(-1.5,1)--(-1.88,1.25)--(-4,4.88)--(-4,7.3)--(4,7.3)-- (4,4.88)--(1.88,1.25)--(1.5,1)--(1.5,3/5)--(0,0);
			
			\draw[->,color=black] (-5.7,0) -- (5.7,0);
			\draw[->,color=black] (0,7.3) -- (0,7.5);
			\draw[color=black] (0,0) -- (0,4);
			\draw[color=white] (0,4) -- (0,7.25);

			\draw [] (0,0) parabola (-5.2,7.14); 
			\draw [] (0,0) parabola (5.2,7.14);
			
			\draw [color=black] (0,0) --(-1.5,3/5);
			\draw [color=black] (0,0) --(1.5,3/5);

			\draw [color=black, dashed] (0,0) --(-2.51,1.67);
			\draw [color=black, dashed] (0,0) --(2.51,1.67);
			\draw [color=black] (-1.5,1)--(-2.51,1.67);
			\draw [color=black] (-1.5,1)--(-1.5,3/5);
			\draw [color=black] (1.5,1)--(2.51,1.67);
			\draw [color=black] (1.5,1)--(1.5,3/5);
			

			\draw [color=black] (-4,4.88)--(-4,7.3);
			\draw [color=black, dashed] (-4,4.88)--(-4,4.2);
			\draw [color=black] (4,4.88)--(4,7.3);
			\draw [color=black, dashed] (4,4.88)--(4,4.2);

			
			
			\draw [color=black] (-1.5,3/5)--(-1.88,1.25);
			
			\draw [color=black] (1.5,3/5)--(1.88,1.25);
			
			\draw [color=gray] (4,4.88)--(1.88,1.25);
			\draw [color=gray] (-4,4.88)--(-1.88,1.25);
			
			\draw [color=black] (4,4.88)--(5,6.6);
			\draw [color=black] (-4,4.88)--(-5,6.6);

			\draw [->,color=black] (-1.88,1.25) arc (220:120:.55cm);
			
			
			\draw [color=black, ->] (1.88,1.25) arc (-40:60:.55cm);
			
			\draw [color=black, ->] (4, 4.88) arc (0:140:.4cm);
			
			
			\draw[->,color=black] (-4,4.88) arc (180:40:.4cm);
			
			

			\draw  (0,0) node [below] {$o$};
			\draw  (-.1,4.1) node [right] {$o'$};
			
			\draw  (1.6,3/5) node [below] {$\gamma_n$};
			\draw  (-1.5,3/5) node [below] {$\gamma_{-n}$};
			
			
			\draw  (5.3,6.8) node [below] {$\gamma_m$};
			\draw  (-5.4,6.8) node [below] {$\gamma_{-m}$};
			\draw  (1.2,1.8) node [above] {$q_{m,n,\epsilon}$};
			\draw  (-.87,1.8) node [above] {$q_{-m,-n,-\epsilon}$};
			\draw  (3.8,5) node [left] {$q'_{m,n,\epsilon}$};
			\draw  (-3.4,5) node [right] {$q'_{-m,-n,-\epsilon}$};
			\draw  (3.95,4.1) node [right] {$\gamma_{m+\epsilon}$};
			\draw  (-3.95,4.1) node [left] {$\gamma_{-m-\epsilon}$};
			\draw  (2.48,1.55) node [right] {$\gamma_{n-\epsilon}$};
			\draw  (-2.42,1.55) node [left] {$\gamma_{-n+\epsilon}$};

			\begin{scriptsize}

			\fill [color=black] (0,0) circle (1.1pt);
			
			\fill [color=black] (-5,6.6) circle (1.1pt);
			\fill [color=black] (5,6.6) circle (1.1pt);
			
			\fill [color=black] (-1.5,3/5) circle (1.1pt);
			\fill [color=black] (-2.51,1.67) circle (1.1pt);
			\fill [color=black] (-1.88,1.25) circle (1.1pt);
			\fill [color=black] (-1.5,1) circle (1.1pt);
			
			\fill [color=black] (1.5,3/5) circle (1.1pt);
			\fill [color=black] (2.51,1.67) circle (1.1pt);
			\fill [color=black] (1.88,1.25) circle (1.1pt);
			\fill [color=black] (1.5,1) circle (1.1pt);

			\fill [color=white] (0,4) circle (1.1pt);

			\fill [color=black] (-4,4.2) circle (1.1pt);
			\fill [color=black] (-4,4.88) circle (1.1pt);
			\fill [color=black] (4,4.2) circle (1.1pt);
			\fill [color=black] (4,4.88) circle (1.1pt);

			\end{scriptsize}
			
			\end{tikzpicture}
			
			\caption{No projection of roots in the grey area}
			
			\label{no hole.2}
			
		\end{centering}
		
	\end{figure}	 
\end{Lem}
\begin{proof} 
	We show that the claimed region is contained in a suitable union of the $U_k$'s. Clearly, it is enough to prove that the open line segment $(\overline{q_{m,n,\epsilon}\;q'_{m,n,\epsilon}})$ is covered completely by a union of the $U_k$'s. Given a number $k \in \frac{1}{2}\mathbb{N}$ where $m < k <n$, the point where the line segments $\overline{\gamma_m\gamma_n}$ and $\overline{o\gamma_k}$ intersect is denoted by $\gamma'_k$, see Figure \ref{no hole.3}.
	\begin{figure} [h]
		\begin{centering}
			\definecolor{zzttqq}{rgb}{0.27,0.27,0.27}
			\definecolor{qqqqff}{rgb}{0.33,0.33,0.33}
			\definecolor{uququq}{rgb}{0.25,0.25,0.25}
			\definecolor{xdxdff}{rgb}{0.66,0.66,0.66}
			
			\begin{tikzpicture}[line cap=round,line join=round,>=triangle 45,x=1.0cm,y=1.0cm]
			
			\filldraw[fill=gray!40!white, draw=white] (0,0) --(1.83,1.5)--(1.83,1.7)--(2.13,2)--(2.13,4)--(0,4)--(0,0)--(0,4)--(-2.13,4)--(-2.13,2)--(-1.83,1.7)--(-1.83,1.5)--(0,0);
			
			\draw[->,color=black] (-3,0) -- (3,0);
			\draw[->,color=black, dashed] (0,4) -- (0,4.2);
			\draw[color=black] (0,0) -- (0,2);
			\draw[color=white] (0,2) -- (0,3.75);
			
			\draw [] (0,0) parabola (-3,4); 
			\draw [] (0,0) parabola (3,4);
			
			\draw[color=black] (1.05,.5) -- (2.81,3.5);
			
			\draw[color=black, dashed] (0,0) -- (1.83,1.5);
			\draw[color=black, dashed] (1.83,4) -- (1.83,1.5);
			
			\draw[color=black, dashed] (0,0) -- (2.13,2);
			\draw[color=black, dashed] (2.13,4) -- (2.13,2);
			
			\draw[color=black] (-1.05,.5) -- (-2.81,3.5);
			
			\draw[color=black, dashed] (0,0) -- (-1.83,1.5);
			\draw[color=black, dashed] (-1.83,4) -- (-1.83,1.5);
			
			\draw[color=black, dashed] (0,0) -- (-2.13,2);
			\draw[color=black, dashed] (-2.13,4) -- (-2.13,2);

			
			\draw [color=black, ->] (1.64,1.53) arc (0:140:.4cm);
			\draw [color=black, ->] (-1.64,1.53) arc (180:40:.4cm);
			
			\draw  (0,0) node [below] {$o$};
			\draw  (0,2) node [right] {$o'$};
			\draw  (1.2,.5) node [below] {$\gamma_n$};
			\draw  (3.1,3.6) node [below] {$\gamma_m$};
			\draw  (-1.2,.5) node [below] {$\gamma_{-n}$};
			\draw  (-3.2,3.6) node [below] {$\gamma_{-m}$};
			
			\draw  (2.3,1.5) node [below] {$\gamma_{k+1/2}$};
			\draw  (2.3,2) node [below] {$\gamma_{k}$};
			
			\draw  (-2.3,1.5) node [below] {$\gamma_{-k-1/2}$};
			\draw  (-2.38,2.1) node [below] {$\gamma_{-k}$};
			\draw  (.7,1.3) node [above] {$\gamma'_k$};
			\draw  (-.6,1.3) node [above] {$\gamma'_{-k}$};
			
			\begin{scriptsize}

			\fill [color=black] (0,0) circle (1.1pt);
			\fill [color=white] (0,2) circle (1.1pt);
			\fill [color=black] (1.05,.5) circle (1.1pt);
			\fill [color=black] (2.81,3.5) circle (1.1pt);
			
			\fill [color=black] (1.83,1.5) circle (1.1pt);
			\fill [color=black] (2.13,2) circle (1.1pt);
			
			\fill [color=black] (-1.05,.5) circle (1.1pt);
			\fill [color=black] (-2.81,3.5) circle (1.1pt);
			
			\fill [color=black] (-1.83,1.5) circle (1.1pt);
			\fill [color=black] (-2.13,2) circle (1.1pt);
			
			\fill [color=black] (1.64,1.53) circle (1.1pt);
			\fill [color=black] (-1.64,1.53) circle (1.1pt);
			
			\end{scriptsize}
			
			\end{tikzpicture}
			
			\caption{Two consecutive points}
			
			\label{no hole.3}
			
		\end{centering}
		
	\end{figure}
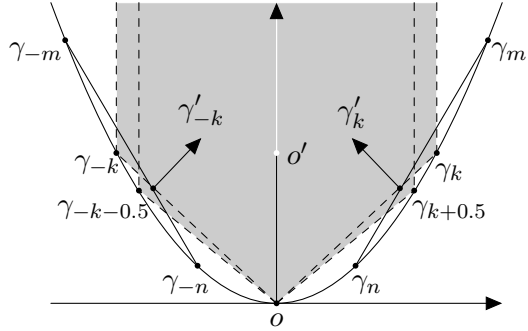
	
	The line segment $(\overline{q_{m,n,\epsilon}\;q'_{m,n,\epsilon}})$ is covered by a union of the $U_k$'s if the point $\gamma'_k$ is inside the area $U_{k+\frac{1}{2}}$ for $m+\epsilon \leq k \leq n-\epsilon-\frac{1}{2}$. 
	The $x$-coordinate of the point $\gamma'_k$ is 
	\begin{equation*}
	x_k = \dfrac{1/mn}{1/m+1/n -1/k}. 
	\end{equation*} 
	One can easily show 
	\begin{equation*}
	f(k) \coloneqq \dfrac{1}{mn}\bigg(k+\dfrac{1}{2}\bigg) + \dfrac{1}{k} \leq \max\left\{f\big(n-\epsilon-1/2\big), f(m+\epsilon) \right\} \overset{(*)}{<} \dfrac{1}{n} + \dfrac{1}{m},
	\end{equation*}
	where $(*)$ follows by the inequality                                                                                                                                                                                                                                                                                                                     \eqref{assumption.1}.
	This gives $x_k <\dfrac{1}{k+\frac{1}{2}}$, so the point $\gamma'_k$ is in $ U_{k+\frac{1}{2}}$. Therefore, the grey region in Figure \ref{no hole.2} is contained in $\bigcup\limits_{m+\epsilon \,\leq \,k \,\leq\, n}U_k\;$ where $k \in \frac{1}{2} \mathbb{N}$ and the claim follows from Lemma \ref{no spherical}.
\end{proof}
\subsection*{Wall and chamber decomposition of $V(X)$} For any coherent sheaf $E$ on $X$, there is a chamber in the subspace of stability conditions $V(X)$ where the notion of Bridgeland stability coincides with the old notion of Gieseker stability which is defined using the Hilbert polynomial \cite[Proposition 14.2]{bridgeland:K3-surfaces}. Recall that the Hilbert polynomial of a coherent sheaf $E$ is defined as
\begin{equation*}
P(E,m) \coloneqq \dfrac{\text{rk}(E)H^2}{2} m^2 + \text{c}(E)H^2 m +\text{s}(E) +\text{rk}(E).
\end{equation*}  
The reduced Hilbert polynomial is $p(E,m) \coloneqq P(E,m)/\alpha(E) $ where $\alpha(E)$ is the leading coefficient of $P(E,m)$. 
\begin{Def}
	A coherent sheaf $E$ on $X$ is called $H$-Gieseker (semi)stable if $E$ is pure (has no subsheaf with lower dimensional support) and for all proper non-trivial subsheaves $F \subset E$, one has $p(F,m) < (\leq) \; p(E,m) $ for $m\gg 0$.
\end{Def}
The notion of slope stability for coherent sheaves on a curve (i.e. an integral separated scheme of dimension one and of finite type over $\mathbb{C}$) is also defined as follows. 
\begin{Def}
	A torsion-free sheaf $F$ on a curve $C$ is slope (semi)stable if for all non-trivial subsheaves $F' \subset F$, we have 
	\begin{equation*}
	\dfrac{\chi(\mathcal{O}_C , F')}{\chi(\mathcal{O}_p , F')} <(\leq) \; \dfrac{\chi(\mathcal{O}_C , F)}{\chi(\mathcal{O}_p , F)},
	\end{equation*}   
	where $\mathcal{O}_p$ is the skyscraper sheaf at the generic point $p$ on the curve $C$ and $\mathcal{O}_C$ is the structure sheaf of $C$.
\end{Def}
Note that $\chi(\mathcal{O}_p,F)$ for the generic point $p$ on the curve $C$ is equal to the rank of the torsion-free sheaf $F$. If we have a closed embedding $i \colon C \hookrightarrow X$ of the curve $C$ into the surface $X$, then the adjoint functors $Li^* \dashv Ri_*$ give $\chi(\mathcal{O}_C,F) = \chi(\mathcal{O}_X,i_*F) = \text{ch}_2(i_*F) = \text{s}(i_*F)$. 
Therefore, the torsion-free sheaf $F$ on $C$ is slope-(semi)stable if and only if $i_*F$ is $H$-Gieseker (semi)stable. The following Lemma introduces the stability conditions that $i_*F$ is stable. 
\begin{Lem}[{\cite[Theorem 3.11]{maciocia:the-walls-projective-spaces}}] \label{large vol}
	Let $F$ be a vector bundle on the curve $C$. 
	\begin{itemize*}
		\item[(a)] If $F$ is slope-(semi)stable, then there exists $w_0 >0$ such that the push-forward $i_*F$ is $\sigma_{(b,w)}$-(semi)stable for any $b \in \mathbb{R}$ and $w >w_0$. 
		\item[(b)] If $i_*F$ is (semi)stable with respect to some stability condition $\sigma_{(b,w)}$, then $F$ is slope-(semi)stable.

	\end{itemize*}
	
\end{Lem}
\begin{proof}
	Any coherent sheaf with rank zero is inside the heart $\mathcal{A}(b)$ for every $b \in \mathbb{R}$. Part $(a)$ follows from \cite[Theorem 3.11]{maciocia:the-walls-projective-spaces} and the fact that
	$\lim\limits_{w \rightarrow \infty} \phi_{(b,w)}(F') = 0$ for any object $F' \in \mathcal{A}(b)$ with rk$(F')>0$. For part $(b)$, suppose $i_*F$ is $\sigma_{(b,w)}$-(semi)stable. Let $F'$ be a subsheaf of $F$, then $i_*F'$ is a subobject of $i_*F$ in the heart $\mathcal{A}(b)$, so 
	\begin{equation*}
	\phi_{(b,w)}(i_*F')  < (\leq)\,  \phi_{(b,w)}(i_*F)  \;\;\; \Rightarrow\;\;\; \dfrac{\text{s}(i_*F')}{\text{c}(i_*F')} < (\leq)\, \dfrac{\text{s}(i_*F)}{\text{c}(i_*F)}.
	\end{equation*}
	This implies $i_*F$ is $H$-Gieseker (semi)stable which gives slope (semi)stability of $F$.  
\end{proof}
For any slope semistable vector bundle $F$ on the curve $C$, the chamber which contains the stability conditions $\sigma_{(b,w)}$ for $w \gg 0$, is called the \emph{Gieseker chamber}. Note that the corresponding point $k(b,w)$ is close to the origin. We use the next lemma to describe regions in $V(X)$ with no walls for a given object.

\begin{Lem}\label{q.3}
	Given a stability condition $\sigma_{(b,w)}$ and an object $E \in \mathcal{D}(X)$ such that 
	\begin{equation*}
	0 < \abs{\text{Im}\, [Z_{(b,w)}(E)]} = \min \bigg\{\abs{\text{Im}\,[Z_{(b,w)}(v')]} \colon v' \in \mathcal{N}(X) \text{  and  } \text{Im} [Z_{(b,w)}(v')] \neq 0 \bigg\},
	\end{equation*}  
	then the stability condition $\sigma_{(b,w)}$ cannot be on a wall for the object $E$. In particular, if $v(E) = (r,cH,s)$ and $b_0=m/n$ for some $m,\, n \in \mathbb{Z}$ such that $nc-mr=\pm 1$, then the stability condition $\sigma_{(b_0,w)}$ cannot be on a wall for $E$.
\end{Lem}
\begin{proof}
	If the stability condition $\sigma_{(b,w)}$ is on a wall $\mathcal{W}_E$, then the object $E$ is strictly $\sigma_{(b,w)}$-semistable. Up to shift, we may assume $E \in \mathcal{A}(b)$, so there are non-trivial objects $E_1, E_2 \in \mathcal{A}(b)$ of the same phase as $E$ such that there is a short exact sequence $E_1 \hookrightarrow E \twoheadrightarrow E_2$. Since $\text{Im}\, [Z_{(b,w)}(E)] \neq 0$, we have $0 < \text{Im}\, [Z_{(b,w)}(E_i)] $ for $i=1,2$ and 
	\begin{equation*}
	\text{Im}\, [Z_{(b,w)}(E)] = \text{Im}\, [Z_{(b,w)}(E_1)] + \text{Im}\, [Z_{(b,w)}(E_2)].
	\end{equation*}
	This is a contradiction to our minimality assumption. If $b_0 = m/n$, then 
	\begin{equation*}
	\text{Im}\, [Z_{(b_0,w)}(E)] = c- \dfrac{m}{n}r = \dfrac{\pm 1}{n}
	\end{equation*} 
	which clearly satisfies the minimality condition.
\end{proof}
We finish this section by the following lemma which introduces stability conditions that a $\mu_H$-stable vector bundle on $X$ is stable with respect to them. 
\begin{Lem}\label{phase one}
	Let $E$ be a $\mu_H$-stable locally free sheaf with Mukai vector $v(E) = (r,cH,s)$. Then $E[1]$ is $\sigma_{(b_0,w)}$-stable of phase one where $b_0 =c/r$. 
\end{Lem}
\begin{proof}
	By definition, $E[1] \in \mathcal{A}(b_0)$ and has phase one which automatically implies it is $\sigma_{(b_0,w)}$-semistable. Assume for a contradiction that $E[1]$ is strictly $\sigma_{(b_0,w)}$-semistable. Let $F_1$ be the $\sigma_{(b_0,w)}$-stable subobject of $E[1]$ and $F_2$ be the quotient 
	\begin{equation}\label{exact.1.}
	F_1 \hookrightarrow E[1] \twoheadrightarrow F_2.
	\end{equation}
	Taking cohomology implies that $H^0(F_2) = 0$, so $F_2 = E'[1]$ for a torsion-free sheaf $E'$. Since $F_1$ is $\sigma_{(b_0,w)}$-stable of phase one, \cite[Lemma 10.1]{bridgeland:K3-surfaces} implies that $F_1$ is a skyscraper sheaf $k(x)$ or shift of a locally free sheaf. The sheaf $E$ is locally free, so Hom$_X(k(x), E[1]) = 0$. Thus $F_1$ must be the shift of a locally free sheaf
	and we have the following exact sequence of coherent sheaves
	\begin{equation*}
	0 \rightarrow H^{-1}(F_1) \rightarrow E \rightarrow E' \rightarrow 0.
	\end{equation*}
	Since Im$\big[Z_{(b,w)}\big(H^{-1}(F_1)\big)\big] = \text{Im}[Z_{(b,w)}(E')] = 0$, the sheaves $E'$ and $H^{-1}(F_1)$ have the same slope as $E$, which contradicts the $\mu_H$-stability of $E$. 
\end{proof}
\section{An upper bound for the number of global sections}\label{section.3}
In this section, we study the \emph{Brill-Noether wall} and introduce an upper bound for the number of global sections of objects in $\mathcal{D}(X)$ depending only on the geometry of their Harder-Narasimhan polygons at a certain limit point, see Proposition \ref{polygon}. \par
We always assume $X$ is a smooth K3 surface with Pic$(X) = \mathbb{Z}.H$. Given an object $E \in \mathcal{D}(X)$, we denote its Mukai vector by $v(E) = \big(\text{rk}(E),\text{c}(E)H,\text{s}(E) \big)$. 
\begin{Lem}\label{new formula}
	Let $E \in \mathcal{A}(0)$ be a $\sigma_{(0,w)}$-semistable object with $\phi_{(0,w)}(E) <1$. Then 
	\begin{equation*}
	v(E)^2 \geq -2 \text{c}(E)^2.
	\end{equation*}
\end{Lem}
\begin{proof}
	Let $0 =\tilde{E}_0 \subset \tilde{E}_1 \subset ... \subset \tilde{E}_{n-1} \subset \tilde{E}_n = E$ be the Jordan-H$\ddot{\text{o}}$lder filtration of $E$ with respect to the stability condition $\sigma_{(0,w)}$. Since the stable factors $E_i = \tilde{E}_i/\tilde{E}_{i-1}$ have the same phase as $E$, we have $ \text{Im}[Z_{(0,w)}(E_i)] = \text{c}(E_i)>0$. Therefore, the length of the filtration $n$ is at most $\text{c}(E)$. Given two factors $E_i$ and $E_j$, we know $\text{Hom}_X(E_i, E_j) = 0 $ if $E_i \not \cong E_j$ and $\text{Hom}_X(E_i, E_i) = \mathbb{C}$. Thus, for any $0 < i,j \leq n$, 
	\begin{equation*}
	\left \langle v(E_i),v(E_j) \right \rangle  = -\text{hom}_X(E_i,E_j) + \text{hom}^1_X(E_i,E_j) -\text{hom}_X(E_j,E_i) \geq -2,
	\end{equation*}
	which implies
	\begin{equation*}
	v(E)^2 = \sum_{i=1}^{n} v(E_i)^2 + 2\sum_{1 \leq i\,<\,j \leq n} \left \langle v(E_i),v(E_j) \right \rangle \geq -2n^2 \geq -2 \text{c}(E)^2.
	\end{equation*}
\end{proof}
A generalization of the argument in \cite[Section 6]{bayer:brill-noether} implies the following lemma. 
 
\begin{Lem}\label{bound for h}\textbf{(Brill-Noether wall)}
	Let $\sigma_{(b_0,w_0)}$ be a stability condition with $b_0<0$ and $k(b_0,w_0)$ sufficiently close to the point $pr\big(v(\mathcal{O}_X)\big) = (0,1)=o'$. Let $E \in \mathcal{D}(X)$ be a $\sigma_{(b_0,w_0)}$-semistable object with the same phase as the structure sheaf $\mathcal{O}_X$. 
	Then 
	\begin{equation}\label{bound for h.in}
	h^0(X,E) \leq \dfrac{\chi(E)}{2} + \dfrac{\sqrt{ \big(\text{rk}(E)-\text{s}(E)\big)^2 +\text{c}(E)^2(2H^2+4) }}{2},
	\end{equation}
	where $h^0(X,E) = \text{dim}_{\;\mathbb{C}}\,\text{Hom}_{\,X}(\mathcal{O}_X,E)$ and $\chi(E) = \text{rk}(E) +\text{s}(E)$ is the Euler characteristic of $E$. 
\end{Lem}
\begin{proof}
	
	We first claim that the structure sheaf $\mathcal{O}_X$ is $\sigma_{(b,w)}$-stable where $k(b,w)$ is sufficiently close to the point $pr\big(v(\mathcal{O}_X)\big) = o'$. By Lemma \ref{phase one}, $\mathcal{O}_X$ is $\sigma_{(0,w)}$-stable where $k(0,w)$ is on the line segment $(\overline{oo'})$, i.e. $w > \sqrt{2/H^2}$.
	Moreover, Lemma \ref{q.3} implies that there is no wall for $\mathcal{O}_X$ passing the line with equation $y=nx$ for any $n \in \mathbb{Z}$, because for any stability condition $\sigma_{(b,w)}$ on such a line, we have $b=\frac{1}{n}$ and  
	\begin{equation*}
	\abs{ \text{Im}[Z_{(b,w)}(\mathcal{O}_X)]} = \dfrac{1}{n} = \min \left\{ \abs{ \text{Im}[Z_{(b,w)}(r,cH,s)]} = \abs{c - \dfrac{r}{n}} \neq 0         \right\}.
	\end{equation*}
	If the object $E$ satisfies c$(E) = 0$, then the projection $pr(v(E))$ lies on the $y$-axis. Remark \ref{plane of having the same phase} implies that $\mathcal{O}_X$ cannot have the same phase as $E$ with respect to $\sigma_{(b_0,w_0)}$ with $b_0 <0$ unless $pr(v(E)) = pr(v(\mathcal{O}_X))$, i.e. $v(E) = kv(\mathcal{O}_X)$. Thus the uniqueness of spherical sheaf with Mukai vector $(1,0,1)$ (see e.g. \cite[Corollary 3.5]{mukai:modili-of-bundles-on-k3-surfaces}) implies that $E$ is the direct sum of $k$-copies of $\mathcal{O}_X$, hence the inequality \eqref{bound for h.in} holds. Thus we may assume c$(E) \neq 0$. Let $L_E$ be the line through $o'$ which passes the point $pr(v(E))$ if s$(E) \neq 0$, or it has slope $\text{rk}(E)/\text{c}(E)$ if s$(E) = 0$. By assumption, the point $k(b_0,w_0)$ is on the line $L_E$. 
	Since we assumed $b_0 <0$, the structure sheaf $\mathcal{O}_X$ is in the heart $\mathcal{A}(b_0)$. This implies $E \in \mathcal{A}(b_0)$ because $E$ is $\sigma_{(b_0,w_0)}$-semistable of the same phase as $\mathcal{O}_X$. When we deform the stability condition $\sigma_{(b_0,w_0)}$ along the line $L_E$ towards the point $o'$, the object $E$ remains in the corresponding heart, by Remark \ref{deforming the stability condition along the wall}. Hence $\lim\limits_{b \rightarrow 0^{-}} \text{Im}[Z_{(b,w)}(E)] \geq 0 $ which implies c$(E) > 0$.  

	We may assume dim $\text{Hom}_X(\mathcal{O}_X, E) \neq 0$, otherwise there is nothing to prove. Consider the evaluation map $\text{ev} \colon \text{Hom}_X(\mathcal{O}_X,E) \otimes \mathcal{O}_X \rightarrow E$.
	As we saw above, the structure sheaf $\mathcal{O}_X$ 
	is $\sigma_{(b_0,w_0)}$-stable, so it is a simple object in the abelian category of semistable objects with the same phase as $\mathcal{O}_X$. Therefore, the morphism ev is injective and the cokernel cok$(\text{ev})$ is also $\sigma_{(b_0,w_0)}$-semistable. 
	Let $\{E_i\}_{i=1}^{i=n}$ be the Jordan-H$\ddot{\text{o}}$lder factors of cok$(\text{ev})$ with respect to the stability condition $\sigma_{(b_0,w_0)}$. By Remark \ref{plane of having the same phase}, the Mukai vector of any factor can be written as $v(E_i)= m_iv(\mathcal{O}_X)+t_iv(E)$ for some $m_i, t_i \in \mathbb{R}$. Note that $\sum_{i=1}^{n}v(E_i) = v(E) - h^0(E)\,v(\mathcal{O}_X)$. \par
	If we deform the stability condition $\sigma_{(b_0,w_0)}$ along the line $L_E$ towards the point $o'$, Remark \ref{deforming the stability condition along the wall} shows that the objects $E_i$ remain stable and of the same phase as $E$ and $\mathcal{O}_X$. Thus, in particular, they remain in the heart and 
	\begin{equation*}
	\lim\limits_{k(b,w) \rightarrow (0^{-},1)} \text{Im}[Z_{(b,w)}(E_i)]= \lim\limits_{b \rightarrow 0^{-}}  [t_i \big(\text{c}(E) -b\,\text{rk}(E)\big) +m_i(-b)] \geq 0.
	\end{equation*}
	This gives $t_i \geq 0$. We have $\sum_{i=1}^{n}c(E_i) = \sum_{i=1}^{n}t_ic(E) = \text{c}(E)$, therefore 
	\begin{equation}\label{sum}
	\sum_{i=1}^{n} t_i =1
	\end{equation}
	If $t_i = 0$, then since $v(E_i)^2 \geq -2$, we have $m_i=1$ so the uniqueness of spherical sheaf again implies $E_i \cong \mathcal{O}_X$. 
	We have c$(E_i) = t_i\text{c}(E) \in \mathbb{Z}$. Combing this with \eqref{sum} proves that the maximum number of factors with $t_i \neq 0$ is equal to $c(E)$. 
	By reordering of the factors, we can assume $E_i \cong \mathcal{O}_X$ for $1 \leq i \leq i_0$ and the other factors satisfy $t_i \neq 0$. Therefore, 
	\begin{equation*}
	v(E)-\big(h^0(X,E)+i_0\big)v(\mathcal{O}_X) = \sum_{i=i_0+1}^{n} w_i
	\end{equation*} 
	where $0 \leq n-i_0 \leq k$. Since $\langle w_i, w_j \rangle \geq -2$ for $1 \leq  i,j \leq n$, the same argument as in Lemma \ref{new formula} implies that 
	\begin{equation*}
	\big(v(E)-(h^0(X,E)+i_0)v(\mathcal{O}_X) \big)^2 = \bigg(\sum_{i=i_0+1}^{n} w_i \bigg)^2 \geq -2\text{c}(E)^2.
	\end{equation*}
	Now solving the quadratic equation 
	\begin{equation*}
	f(x) = \big( v(E)-xv(\mathcal{O}_X) \big)^2 + 2k^2 = -2x^2 + 2x\, \chi(E) +v(E)^2+2\text{c}(E)^2 =0
	\end{equation*}
	shows that
	\begin{equation}\label{better bound in lemm}
	h^0(X,E)\leq h^0(X,E) +i_0 \leq \dfrac{\chi(E)}{2} + \dfrac{\sqrt{ \big(\text{rk}(E)-\text{s}(E)\big)^2 +\left(2H^2+4\right)\text{c}(E)^2 }}{2}.
	\end{equation}
\end{proof}

\begin{Def}
	Given a stability condition $\sigma_{(b,w)}$ and an object $E \in \mathcal{A}(b)$, the Harder-Narasimhan polygon of $E$ 
	is the convex hull of the points $Z_{(b,w)}(E')$ for all subobjects $E '\subset E$ of $E$. 
\end{Def}
If the Harder-Narasimhan filtration of $E$ is the sequence
\begin{equation*}
0 = \tilde{E}_0 \subset \tilde{E}_1 \subset .... \subset \tilde{E}_{n-1} \subset \tilde{E}_n =E,
\end{equation*}
then the points $\left\{ p_i = Z_{(b,w)}(\tilde{E}_i) \right\}_{i=0}^{i=n}$ are the extremal points of the Harder-Narasimhan polygon of $E$ on the left side of the line segment $\overline{oZ_{(b,w)}(E)}$, see Figure \ref{polygon figure.1}. 


\begin{figure} [h]
	\begin{centering}
		\definecolor{zzttqq}{rgb}{0.27,0.27,0.27}
		\definecolor{qqqqff}{rgb}{0.33,0.33,0.33}
		\definecolor{uququq}{rgb}{0.25,0.25,0.25}
		\definecolor{xdxdff}{rgb}{0.66,0.66,0.66}
		
		\begin{tikzpicture}[line cap=round,line join=round,>=triangle 45,x=1.0cm,y=1.0cm]
		
		\draw[->,color=black] (-2.3,0) -- (2.3,0);
		
		\filldraw[fill=gray!40!white, draw=white] (0,0) --(-1.3,.4)--(-1.7,1)--(-1.3,1.8)--(.5,2.5)--(2,2.5)--(2,0);

		\draw [ color=black] (0,0)--(-1.3,.4);
		\draw [color=black] (-1.3,.4)--(-1.7,1);
		\draw [color=black] (-1.7,1)--(-1.3,1.8);
		\draw [color=black] (-1.3,1.8)--(.5,2.5);
		\draw [color=black] (2,2.5) -- (.5,2.5);
		\draw [color=black, dashed] (0,0) -- (.5,2.5);
		
		\draw[->,color=black] (0,0) -- (0,3.3);

		\draw (2.3,0) node [right] {Re$[Z_{(b,w)}(-)]$};
		\draw (0,3.3) node [above] {Im$[Z_{(b,w)}(-)]$};
		\draw (0,0) node [below] {$o$};
		\draw (-1.3,.4) node [left] {$p_1$};
		\draw (-1.7,1) node [left] {$p_2$};
		\draw (-1.3,1.8) node [above] {$p_3$};
		\draw (0,2.8) node [right] {$p_4 = Z_{(b,w)}(E)$};

		\begin{scriptsize}
		
		\fill [color=black] (-1.3,.4) circle (1.1pt);
		\fill [color=black] (-1.7,1) circle (1.1pt);
		\fill [color=black] (-1.3,1.8) circle (1.1pt);
		\fill [color=black] (.5,2.5) circle (1.1pt);
		\fill [color=uququq] (0,0) circle (1.1pt);
		
		\end{scriptsize}
		
		\end{tikzpicture}
		
		\caption{The HN polygon is in the grey area.} 
		
		\label{polygon figure.1}
		
	\end{centering}
	
\end{figure}
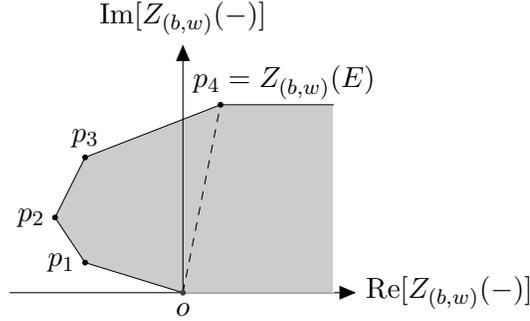

We define the following non-standard norm on $\mathbb{C}$: 
\begin{equation}\label{definition of norm}
\lVert x+iy \rVert = \sqrt{x^2 + (2H^2+4)y^2}.
\end{equation} 
For two points $p$ and $q$ on the complex plane, the length of the line segment $\overline{pq}$ induced by the above norm is denoted by $\lVert \overline{pq} \rVert$. The function $\overline{Z} \colon K(X) \rightarrow \mathbb{C}$ is defined as $$\overline{Z}(E) = Z_{\left(0,\sqrt{2/H^2}\right)}(E) =  \text{rk}(E)-\text{s}(E) \,+\, i\,\text{c}(E).$$ The next proposition shows that we can bound the number of global sections of an object in $\mathcal{A}(0)$ via the length of the Harder-Narasimhan polygon at some limit point.  
\begin{Prop}\label{polygon}
	Consider an object $E \in \mathcal{A}(0)$ which has no subobject $F \subset E$ in $\mathcal{A}(0)$ with ch$_1(F) = 0$. 
	\begin{itemize*}
		\item[(a)] There exists $w^* > \sqrt{2/H^2}$ such that the Harder-Narasimhan filtration of $E$ is a fixed sequence
		\begin{equation*}\label{HN}
		0 = \tilde{E}_{0} \subset \tilde{E}_{1} \subset .... \subset \tilde{E}_{n-1} \subset  \tilde{E}_n=E,
		\end{equation*}
		for all stability conditions $\sigma_{(0,w)}$ where $\sqrt{2/H^2}< w < w^*$.
		\item[(b)] Let $p_i \coloneqq  \overline{Z}(\tilde{E_i})$ for $0 \leq i \leq n$, then
		\begin{equation*}
		h^0(X,E) \leq \dfrac{\chi(E)}{2}  + \dfrac{1}{2} \sum_{i=1}^{n} \lVert \overline{p_ip_{i-1}} \rVert.
		\end{equation*} 
	\end{itemize*} 	
\end{Prop}
\begin{proof}
	We first show that there exists $ w_1 > \sqrt{2/H^2}$ such that the semistable factor $\tilde{E}_1$ is fixed for the stability conditions of form $\sigma_{(0,w)}$ where $\sqrt{2/H^2} < w < w_1$. Let $\sigma_{(0,w)}$ be a stability condition such that $\sqrt{2/H^2} < w < \sqrt{4/H^2} \coloneqq w_0$ and $v_1 = (r_1,c_1H,s_1)$ be a possible class of the semistable factor $\tilde{E}_1$. We have $0 < \text{Im}[Z_{(0,w)}(\tilde{E}_1)] = c_1 < \text{Im}[Z_{(0,w)}(E)] = c(E)$. Lemma \ref{new formula} implies that 
	\begin{equation*}\label{first.11}
	r_1s_1 \leq c_1^2 \bigg( \dfrac{H^2}{2}+1 \bigg) \leq \text{c}(E)^2 \bigg( \dfrac{H^2}{2}+1 \bigg).
	\end{equation*}
	Hence, if $r_1s_1 >0$, there are only finitely many possibilities for the class $v_1$. 
	If $r_1 \geq 0$ and $s_1 \leq 0$, then since $\phi_{(0,w)}(v(E)) \leq \phi_{(0,w)}(v_1)$, we have
	\begin{equation*}\label{first.12}
	\max\{r_1,-s_1\} \leq  \frac{r_1H^2w^2}{2} -s_1 = \text{Re}[Z_{(0,w)}(v_1)] \leq \max\, \{\text{Re} [Z_{(0,w)}\big(v(E)\big)] , 0 \},
	\end{equation*}  
	and if $r_1 \leq 0 $ and $s_1 \geq 0$, the existence of HN filtration for $E$ at $\sigma_{(0,w_0)}$ implies that there exists a real number $M_0$ such that 
	\begin{equation*}\label{first.13}
	M_0 \leq \text{Re}[Z_{(0,w_0)}(v_1)] = \frac{r_1H^2w_0^2}{2} -s_1 \leq r_1-s_1.
	\end{equation*} 
	Thus in any case, there are only finitely many possibilities for $v_1$. Note that the heart $\mathcal{A}(0)$ for the stability conditions $\sigma_{(0,w)}$ is fixed and does not depend on $w$. Moreover, the ordering of the phase function $\phi_{(0,w)}$ is the same as the ordering given by the linear function $-\frac{\text{Re}[Z_{(0,w)}(-)]}{\text{Im}[Z_{(0,w)}(-)]}$. Let $\tilde{E}_1$ be the semistable subobject of $E$ of maximum phase in the HN filtration of $E$ with respect to $\sigma_{(0,w)}$ where $w=w_0$. When we decrease $w$, the subobject $\tilde{E}_1$ in the HN filtration changes if another subobject of $E$ gets bigger phase. Since there are only finitely many possibilities for the class of these subobjects which achieve the maximum phase, there is $w_1 > \sqrt{2/H^2}$ such that the subobject $\tilde{E}_1$ is fixed in the HN filtration of $E$ with respect to $\sigma_{(0,w)}$ when $w \in \big(\sqrt{2/H^2}, w_1\big)$.  \par
	Continuing this argument by induction, one shows that there is a number $ w_{i} $ such that $\sqrt{2/H^2} < w_i < w_{i-1}$ and the semistable factor $E_i = \tilde{E}_i/\tilde{E}_{i-1}$ which is the semistable subobject of $E/\tilde{E}_{i-1}$ with the maximum phase, is fixed for the stability conditions $\sigma_{(0,w)}$ where $ \sqrt{2/H^2} < w < w_i$. Note that $0 < \text{Im}[Z_{(0,w)} (E_i)] < \text{c}(E)$, so the length of the HN filtration of $E$ is at most $\text{c}(E)$. This completes the proof of $(a)$.\par
	Since $\text{c}(E_i) \neq 0$ for $1 \leq i \leq n$, the point $pr(v(E_i))$ is not on the $y$-axis. Proposition \ref{line wall}, part $(d)$ implies that the line segment $\overline{oo'}$ is not a wall for the semistable factor $E_i$. Moreover, since $E_i$ is $\sigma_{(0,w)}$-semistable for $\sqrt{2/H^2}< w <w^*$, there is no wall for $E_i$ which passes the line segment $\overline{o'o^*}$ where $o^* = k(0,w^*)$. In other words, these stability conditions are all inside one chamber for $E_i$. If $\text{s}(E_i) \neq 0$, we define $V_i$ as a cone in $\mathbb{R}^2$ with two rays $\overline{pr_io'}$ and $\overline{pr_io^*}$ where $pr_i  \coloneqq pr(v(E_i))$ and if $\text{s}(E_i) = 0$, then $V_i$ is defined as the area between two parallel lines of slope $\text{rk}(E_i)/\text{c}(E_i)$ which pass through the points $o'$ and $o^*$, see Figure \ref{HN}.\par
	Lemma \ref{no spherical} implies that there is a small rectangle $a_1a_2a_3a_4$ around the point $o'=(0,1)$ such that there is no projection of roots other than $pr(v(\mathcal{O}_X))$ inside it.
	Let $V'_i$ be the intersection of $V_i$ and the rectangle $a_1a_2a_3a_4$, see the dashed area in Figure \ref{HN}. The structure of the wall and chamber decomposition implies that $E_i$ is semistable with respect to the stability conditions in $V'_i$. In particular, it is $\sigma_i \coloneqq \sigma_{(\tilde{b}_i,\tilde{w}_i)}$-semistable where $\sigma_i$ is on the top boundary of $V'_i$, i.e. the associated point $k(\tilde{b}_i,\tilde{w}_i)$ is on the top boundary of $V'_i$. In the figures, by abuse of notation, we denote the point $k(\tilde{b}_i,\tilde{w}_i)$ by the corresponding stability condition $\sigma_i$.  \par
	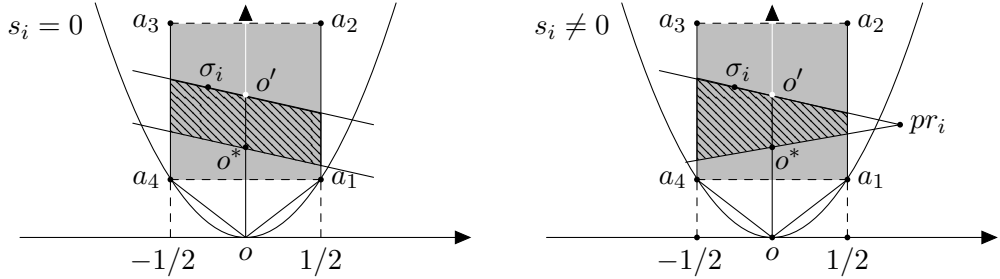
\begin{figure} [h]
		\begin{centering}
			
			\begin{tikzpicture}[line cap=round,line join=round,>=triangle 45,x=1.0cm,y=1.0cm]
			
			\filldraw[fill=gray!50!white, draw=white] (.7,.8)--(.7,2.85)--(-.7,2.85)--(-.7,.8);
			
			\filldraw[fill=gray!50!white, draw=white] (-6.3,.8)--(-6.3,2.85)--(-7.7,2.85)--(-7.7,.8);

			\draw[->,color=black] (-10,0) -- (-4,0);
			\draw[->,color=black] (-3,0) -- (3,0);

			\draw [] (0,0) parabola (2,3.12); 
			\draw [] (0,0) parabola (-2,3.12); 
			\draw [] (-7,0) parabola (-9,3.12); 
			\draw [] (-7,0) parabola (-5,3.12);

			\draw [color=black, dashed] (-.7,.8)--(-.7,2.85);
			\draw [color=black, dashed] (.7,.8)--(.7,2.85);
			
			\draw [color=black, dashed] (-7.7,.8)--(-7.7,2.85);
			\draw [color=black, dashed] (-6.3,.8)--(-6.3,2.85);

			\draw (-3.3,2.8) node [right] {$s_i \neq 0$};
			\draw (-10.3,2.8) node [right] {$s_i = 0$};

			\draw (.63,.9) node [right] {$a_1$};
			\draw (.7,2.85) node [right] {$a_2$};
			\draw (-.65,.9) node [left] {$a_4$};
			\draw (-.7,2.85) node [left] {$a_3$};
			
			\draw (-6.33,.9) node [right] {$a_1$};
			\draw (-6.3,2.85) node [right] {$a_2$};
			\draw (-7.65,.9) node [left] {$a_4$};
			\draw (-7.7,2.85) node [left] {$a_3$};

			\draw [color=black] (0,0)--(0,1.9);
			\draw [color=white] (0,1.9)--(0,2.85);
			\draw[dashed, ->,color=black] (0,2.9) -- (0,3.1);
			
			\draw [dashed] (-.7,.8)--(.7,.8); 
			\draw [dashed] (-.7,2.85)--(.7,2.85);
			
			\draw [color=black] (-7,0)--(-7,1.9);
			\draw [color=white] (-7,1.9)--(-7,2.85);
			\draw[dashed, ->,color=black] (-7,2.9) -- (-7,3.1);
			\draw [dashed] (-7.7,.8)--(-6.3,.8); 
			\draw [dashed] (-7.7,2.85)--(-6.3,2.85);

			\draw (0,2.1) node [right] {$o'$};
			\draw (-.1,1.05) node [right] {$o^*$};
			\draw (1.7,1.5) node [right] {$pr_i$};
			\draw [color=black] (1.7,1.5)--(-1.5,2.22);
			\draw [color=black] (1.7,1.5)--(-1.15,1);
			\draw (0,0) node [below] {$o$};
			\draw (-.45,1.95) node [above] {$\sigma_i$};
			
			\draw[pattern=north west lines, pattern color=black, thin] (-.7 ,1.1)--(-.7,2.04)--(.7,1.73)--(.7,1.35);
			
			\draw (-7,2.1) node [right] {$o'$};
			\draw (-6.9,1.05) node [left] {$o^*$};
			\draw [color=black] (-5.3,1.5)--(-8.5,2.22);
			\draw [color=black] (-5.3,.8)--(-8.5,1.52);
			\draw (-7,0) node [below] {$o$};
			\draw (-7.45,1.95) node [above] {$\sigma_i$};

			\draw[pattern=north west lines, pattern color=black, thin] (-7.7 ,1.35)--(-7.7,2.05)--(-6.3,1.7)--(-6.3,1.05);

			\begin{scriptsize}
			
			\fill [color=white] (0,1.9) circle (1.1pt);
			\fill [color=black] (.7,.8) circle (1.1pt);
			\fill [color=black] (-.7,.8) circle (1.1pt);
			\fill [color=black] (0,0) circle (1.1pt);
			
			\fill [color=black] (0,1.2) circle (1.1pt);
			
			
			\fill [color=black] (-.7 ,2.85) circle (1.1pt);
			\fill [color=black] (.7,2.85) circle (1.1pt);

			\fill [color=black] (1.7,1.5) circle (1.1pt);
			\fill [color=black] (-.5,2) circle (1.1pt);

			\fill [color=black] (-7.7 ,2.85) circle (1.1pt);
			\fill [color=black] (-6.3,2.85) circle (1.1pt);
			\fill [color=black] (-7.7 ,.8) circle (1.1pt);
			\fill [color=black] (-6.3,.8) circle (1.1pt);

			\fill [color=white] (-7,1.9) circle (1.1pt);
			\fill [color=black] (-7,1.2) circle (1.1pt);
			\fill [color=black] (-7.5,2) circle (1.1pt);
			

			\end{scriptsize}
			
			\end{tikzpicture}
			
			\caption{The object $E_i$ remains semistable when we go to $\sigma_i$}
			
			\label{HN}
			
		\end{centering}
		
	\end{figure}
	
	We may assume $-1 \ll \tilde{b}_i <0$. Since $E_i \in \mathcal{A}(0)$ is of phase less than one, it remains in the heart $\mathcal{A}(\tilde{b}_i)$. By Remark \ref{plane of having the same phase}, the objects $E_i$ and $\mathcal{O}_X$ have the same phase with respect to $\sigma_i$. Hence Lemma \ref{bound for h} gives 
	\begin{align*}
	h^0(X,E_i) \leq \dfrac{\text{rk}(E_i)+\text{s}(E_i) + \sqrt{\big(\text{rk}(E_i)-\text{s}(E_i)\big)^2 + \text{c}(E_i)^2(2H^2+4)}}{2}= \;\;\;\;\;\;\;\;\;\;\;\;\;\\
	\;\;\;\;\dfrac{\text{rk}(E_i)+\text{s}(E_i)}{2}\; +\; \dfrac{\lVert \overline{p_ip_{i-1}} \rVert}{2}.\;\;\;
	\end{align*}
	Thus,
	\begin{equation*}
	h^0(X,E) \leq \sum_{i=1}^{n} h^0(X,E_i) \leq \dfrac{\chi(E)}{2} + \dfrac{1}{2} \sum_{i=1}^{n} \lVert \overline{p_ip_{i-1}}\rVert.  
	\end{equation*} 
\end{proof}
We denote by $P_E$ the convex hull of the points $\{p_0, p_1, ..., p_n\}$ as defined in Proposition \ref{polygon}, part $(b)$. We think of $P_E$ as the Harder-Narasimhan polygon of $E$ on the left at the limit point. We finish this section by stating two useful inequalities which are the result of deformation
of stability conditions.
\begin{Lem}\label{bound for slope}
	Consider a stability condition $\sigma_{(b_0,w_0)}$ and a $\sigma_{(b_0,w_0)}$-semistable object $E \in \mathcal{A}(b_0)$ with $\text{c}(E) \neq 0$. Let $L$ be a line through the point $pr(v(E))$ if $\text{s}(E) \neq 0$ or it has a slope of $\text{rk}(E)/\text{c}(E)$ if $\text{s}(E) =0$. Let $q_1 = (x_1,y_1)$ and $q_2 = (x_2,y_2)$ be two points on the line $L$ where $y_1y_2 \neq 0$ and $x_1/y_1 \leq x_2/y_2$. Suppose the point $k(b_0,w_0)$ is on the open line segment $(\overline{q_1q_2})$. If every point on the open line segment $(\overline{q_1q_2})$ is in correspondence to a stability condition, i.e., if $(\overline{q_1q_2})\subset V(X)$, then   	 
	\begin{equation}\label{points}
	\mu_H^{+}\big(H^{-1}(E)\big) \leq \dfrac{x_1}{y_1} \;\;\;\;\;  \text{and} \;\;\;\;\; 	  \dfrac{x_2}{y_2} \leq \mu^{-}_H\big(H^{0}(E)\big).
	\end{equation}
\end{Lem}
\begin{proof}
	Remark \ref{deforming the stability condition along the wall} implies that $E$ is $\sigma_{(b,w)}$-semistable and it is in the heart $\mathcal{A}(b)$ whenever the point $k(b,w)$ is on the open line segment $(\overline{q_1q_2})$. Therefore,  
	\begin{equation*}
	\mu_H^{+}\big(H^{-1}(E)\big) \leq b < \mu^{-}_H\big(H^{0}(E)\big).
	\end{equation*}
	If $k(b_i,w_i) = q_i$ , then $b_i = x_i/y_i$. Thus the stability conditions close to the points $q_1$ or $q_2$ give the inequalities \eqref{points}. 
\end{proof}

\section{The Brill-Noether loci}\label{section.4}
In this  section, we first show that the morphism $\psi \colon M_{X,H}(v) \rightarrow \mathcal{BN}$ described in \eqref{function} is well-defined. Then we consider a slope semistable rank $r$-vector bundle $F$ on the curve $C$ of degree $2rs$ and describe the location of the wall that bounds the Gieseker chamber for the push-forward of $F$. Finally, in Proposition \ref{prop.2}, we show that if the number of global sections of $F$ is high enough, then it must be the restriction of a vector bundle on the surface.\par 
We assume throughout Section \ref{section.4} that $X$ is a K3 surface with Pic$(X) = \mathbb{Z}.H$ and $H^2 = 2rs$ for some $r \geq 2$ and $s \geq \max\{r,5\}$. We also assume $C$ is a curve in the linear system $|H|$ and $i \colon C \hookrightarrow X$ is the embedding of the curve $C$ into the surface $X$. The push-forward of a rank $r$-vector bundle $F$ on the curve $C$ of degree $2rs$ has Mukai vector $v(i_*F) = (0,rH,2rs-r^2s)$. 
Let $M_{X,H}(v)$ be the moduli space of $H$-Gieseker semistable sheaves on the surface $X$ with Mukai vector $v = (r,H,s)$. Since $v^2 =0$ and $v$ is primitive, the moduli space $M_{X,H}(v)$ is a smooth projective K3 surface \cite[Proposition 10.2.5 and Corollary 10.3.5]{huybretchts:lectures-on-k3-surfaces}. Any coherent sheaf $E \in M_{X,H}(v)$ is a $\mu_H$-stable locally free sheaf \cite[Remark 6.1.9]{huybrechts:geometry-of-moduli-space-of-sheaves}. Note that $E(-H)$ is also $\mu_H$-stable. Let $u \coloneqq v(i_*F) - v = v(E(-H)[1])$,
\begin{equation*}
p_v \coloneqq pr(v) = \bigg(\dfrac{1}{s} , \dfrac{r}{s} \bigg) \;\;\; \text{and} \;\;\;
p_u \coloneqq pr(u) = \bigg( \dfrac{-1}{s(r-1)} , \dfrac{r}{s(r-1)^2} \bigg).
\end{equation*} 
We also denote by $\tilde{o}$ the point at which the line segments $\overline{p_up_v}$ and $\overline{o'o}$ intersect, where $o' = pr(v(\mathcal{O}_X))$. 
Define the object $K_E \in \mathcal{D}(X)$ as the cone of the evaluation map:
\begin{equation}\label{cone}
\mathcal{O}_X^{h^0(X,E)} \xrightarrow{\text{ev}_E} E \rightarrow K_E.
\end{equation} 
Denote the point $\big(-1/r , s/r\big)$ by $q$. Lemma \ref{no spherical.2} for $m=r$, $n= s(r-1)$ and $\epsilon =1$ implies that there is no projection of roots in the grey area and on the open line segment $(\overline{et})$ in Figure \ref{no hole.5}, where $e = q_{-m,-n,-\epsilon}$ and $t=q'_{-m,-n,-\epsilon}$. As before, we denote by $\gamma_n$ the point on the parabola $y=rsx^2$ with the $x$-coordinate $1/n$.
\begin{figure} [h]
	\begin{centering}
		\definecolor{zzttqq}{rgb}{0.27,0.27,0.27}
		\definecolor{qqqqff}{rgb}{0.33,0.33,0.33}
		\definecolor{uququq}{rgb}{0.25,0.25,0.25}
		\definecolor{xdxdff}{rgb}{0.66,0.66,0.66}
		
		\begin{tikzpicture}[line cap=round,line join=round,>=triangle 45,x=1.0cm,y=1.0cm]
		
		\filldraw[fill=gray!40!white, draw=white] (0,0) --(-1.5,3/5)--(-1.5,1)--(-1.89,1.26)--(-4,4.88)--(-4,7.3)--(3,7.3)-- (3,12/5)--(0,0);
		
		\draw[->,color=black] (-5.7,0) -- (5.7,0);
		\draw[color=black] (0,0) -- (0,4);
		\draw[color=white] (0,4) -- (0,7.3);
		\draw[->,color=black] (0,7.3) -- (0,7.5);
		
		\draw [] (0,0) parabola (-5.2,7.14); 
		\draw [] (0,0) parabola (5.2,7.14);
		
		\draw [color=black] (0,0) --(-1.5,3/5);

		\draw [color=black] (-1.5,1)--(-1.5,3/5);
		\draw [color=black] (-1.5,1)--(-2.53,1.69);
		
		\draw [color=black] (-4,4.88)--(-4,7.3);
		\draw [color=black, dashed] (-4,4.88)--(-4,4.2);
		
		\draw [color=black, dashed] (0,0)--(-5,6.6);
		\draw [color=black, dashed] (3,12/5)--(-5,6.6);
		\draw [color=black] (3,12/5)--(0,0);
		\draw [color=black] (3,12/5)--(3,7.3);

		\draw [color=black, dashed] (3,12/5)--(-1.5,3/5);

		\draw [color=blue] (0,0)--(-3.7,4.35);
		\draw [color=blue] (0,0)--(-2.32,2);
		

		
		\draw[->,color=black] (-5.7,0) -- (5.7,0);

		\draw (0,7.5) node [above] {$y$};
		\draw  (5.7,0) node [right] {$x$};
		\draw (0,0) node [below] {$o$};
		\draw  (1.4,3.8) node [above] {$o' = pr\big(v(\mathcal{O}_X)\big)$};
		\draw  (3,12/5) node [right] {$p_v$};
		\draw  (-1.8,.8) node [below] {$p_u$};
		\draw  (-4.95,6.6) node [left] {$q$};
		
		\draw  (-3.85,4.55) node [right] {$t'$};
		\draw  (-2.2,1.95) node [above] {$e'$};
		\draw (-.9,4.4) node [above] {$\sigma_3$};
		\draw   (-.9,.27) node [above] {$\sigma_1$};
		\draw  (-1.8,2.6) node [above] {$\sigma_4$};
		\draw  (-4.45,4.2) node [below] {$\gamma_{-r-1}$}; 
		\draw  (-3.95,4.91) node [left] {$t$};
		\draw  (-1.83,1.175) node [above] {$e$};
		
		\draw   (-2.48,1.6) node [left] {$\gamma_{-s(r-1)+1}$};
		\draw   (0.15,1.2) node [above] {$\tilde{o}$};
		
		\draw [color=black] (-1.5,3/5)--(-1.89,1.26);
		\draw [color=gray] (-1.89,1.26)--(-4,4.88);
		\draw [color=black] (-4,4.88)--(-5,6.6);

		\begin{scriptsize}

		\fill [color=black] (0,0) circle (1.1pt);
		
		\fill [color=black] (-1.5,3/5) circle (1.1pt);

		\fill [color=black] (-5,6.6) circle (1.1pt);
		
		\fill [color=black] (-2.53,1.69) circle (1.1pt);
		\fill [color=black] (-1.89,1.26) circle (1.1pt);
		
		\fill [color=black] (-4,4.2) circle (1.1pt);
		\fill [color=black] (3,12/5) circle (1.1pt);
		
		\fill [color=white] (0,4) circle (1.7pt);
		
		\fill [color=black] (-4,4.88) circle (1.1pt);
		
		\fill [color=gray] (-3.7,4.35) circle (1.1pt);
		\fill [color=gray] (-2.32,2) circle (1.1pt);

		\fill [color=black] (-1,4.5) circle (1.1pt);
		\fill [color=black] (-2,2.65) circle (1.1pt);
		
		\fill [color=black] (-.8,.33) circle (1.1pt);
		\fill [color=black] (0,1.2) circle (1.1pt);
		\end{scriptsize}
		
		\end{tikzpicture}
		
		\caption{No projection of roots in the grey area}
		
		\label{no hole.5}
		
	\end{centering}
	
\end{figure} 
\begin{Prop}\label{prop.1}
	Let $E \in M_{X,H}(v)$ be a $\mu_H$-stable vector bundle on the surface $X$. Then we have
	\begin{itemize}
		\item[(a)] Hom$_X\big(E,E(-H)[1]\big) = 0$.
		\item[(b)] The restriction $E|_C$ is a slope stable vector bundle on the curve $C$ and $h^0(C,E|_C) = r+s$. In particular, the morphism $\psi$ described in \eqref{function} is well-defined.
		\item[(c)] The object $K_E$ is of the form $K_E=E'[1]$ where $E'$ is a $\mu_H$-stable locally free sheaf on $X$ and $\text{Hom}_X\big(E', E(-H)[1]\big) = 0$.
	\end{itemize}
\end{Prop}
\begin{proof}
	The objects $E$ and $E(-H)[1]$ have the same phase with respect to the stability condition $\tilde{\sigma} \coloneqq \sigma_{(0,\tilde{w})}$ where $k(0,\tilde{w}) = \tilde{o}$, see Figure \ref{no hole.5}. There are no homomorphisms between non-isomorphic stable objects of the same phase. Hence to prove claim $(a)$, we only need to show both $E$ and $E(-H)[1]$ are $\tilde{\sigma}$-stable. 
	
	By \cite[Proposition 14.2]{bridgeland:K3-surfaces}, the $\mu_H$-stable sheaf $E$ is $\sigma_{(0,w)}$-stable where $w\gg 0$. Lemma \ref{q.3} implies that there is no wall for $E$ intersecting the line segment $(\overline{oo'})$. Thus $E$ is $\tilde{\sigma}$-stable. Lemma \ref{phase one} implies that $E(-H)[1]$ with Mukai vector $\big(-r,(r-1)H,-s(r-1)^2\big)$ is $\sigma_1\coloneqq \sigma_{(b_1,w_1)}$-stable where $b_1=-(r-1)/r$ and $w_1$ is arbitrary. Let $e'$ be the point that the line segment $\overline{qp_u}$ intersects the line $x=b_2y$ where
	\begin{equation*}
	b_2 = - \dfrac{r-2}{r-1} \;\;\;\; \text{if} \; r>2 \;\;\;\;\;\;\; \text{or} \;\;\;\;\;\;\; b_2=-\dfrac{1}{3} \;\;\;\; \text{if} \; r=2. 
	\end{equation*}
	If $s \geq \max\{r,5\}$, then
	\begin{equation*}
	-\dfrac{s(r-1)-1}{rs} < b_2 < -\dfrac{r+1}{rs}.
	\end{equation*}
	Thus the line segment $\overline{oe'}$ is located between two lines $\overline{o\gamma_{-r-1}}$ and $\overline{o\gamma_{-s(r-1)+1}}$ and it is on the grey area with no projection of roots. We claim that there is no wall for $E(-H)[1]$ intersecting the line segment $(\overline{oe'}]$. Consider a stability condition of form $\sigma_{(b_2,w)}$ where the point $k(b_2,w)$ is on $(\overline{oe'}]$. If $r>2$, then
	\begin{equation*}
	\abs{\text{Im}[Z_{(b_2,w)}(E(-H))]} = \frac{1}{r-1} = \min \left\{ \abs{\text{Im}[Z_{(b_2,w)}(r',c'H,s')] }= \abs{c'+\frac{r-2}{r-1}r'} \neq 0       \right\},
	\end{equation*}
	and if $r=2$, then $\abs{\text{Im}[Z_{(b_2,w)}(E(-H))]} = \frac{1}{3}$. Thus in any case, the minimality condition of Lemma \ref{q.3} is satisfied. Hence the stability condition $\sigma_{(b_2,w)}$ cannot be on a wall for $E(-H)[1]$. This in particular implies $E(-H)[1]$ is $\tilde{\sigma}$-stable which finishes the proof of claim $(a)$.\par 
	Consider the short exact sequence 
	\begin{equation}\label{sec}
	E \hookrightarrow i_*E|_C \twoheadrightarrow E(-H)[1]. 
	\end{equation}
	Since $E$ and $E(-H)[1]$ are $\tilde{\sigma}$-stable of the same phase, their extension $i_*E|_C$ is $\tilde{\sigma}$-semistable and the objects $E$ and $E(-H)[1]$ are its JH factors. We have $\phi_{(0,w)}(E) < \phi_{(0,w)}(i_*E|_C)$ for $w> \tilde{w}$, thus the object $i_*E|_C$ is $\sigma_{(0,w)}$-stable and by Lemma \ref{large vol}, $E|_C$ is slope stable. 
	
	The next step is to show $h^0(X,E) =r+s$. Consider a stability condition of form $\sigma_{(b_3,w_3)} \eqqcolon \sigma_3$ such that the point $k(b_3,w_3)$ is on the line segment $(\overline{qo'})$ and it is sufficiently close to the point $o'$. Lemma \ref{q.3} implies that there is no wall for $E$ which intersects the open line segment $(\overline{oo'})$. Therefore $\sigma_{(0,w)}$-stability of $E$ for $w \gg 0$ implies that it is $\sigma_3$-semistable. Moreover, $E$ has the same phase as the structure sheaf $\mathcal{O}_X$ with respect to $\sigma_3$, thus Lemma \ref{bound for h} implies that 
	$$h^0(X,E) \leq \left\lfloor \frac{r+s}{2} + \frac{\sqrt{(r+s)^2 +4}}{2} \right\rfloor = r+s.$$ The coherent sheaf $E$ is $\mu_H$-stable and has positive slope, so Hom$_X(E,\mathcal{O}_X) =0$ and 
	\begin{equation*}
	\chi(E) = r+s = h^0(X,E) - h^1(X,E).
	\end{equation*}
	Therefore $h^0(X,E) = r+s$ and the object $K_E$ has Mukai vector $v(K_E) = (-s,H,-r)$. On the other hand, since there is no wall for $E(-H)[1]$ which passes the line segment $(\overline{oe'}]$, it is stable with respect to the stability conditions on the line segment $(\overline{p_uo'})$, where $\mathcal{O}_X$ is also stable and has the same phase as $E(-H)[1]$. Thus
	\begin{equation*}\label{zero hom}
	\text{Hom$_X(\mathcal{O}_X , E(-H)[1]) = 0$}.
	\end{equation*}
	Thus the short exact sequence \eqref{sec} gives $h^0(C,E|_C) = h^0(X,E) = r+s$, which completes the proof of $(b)$. 
	
	The sheaves $E$ and $\mathcal{O}_X$ are $\sigma_3$-semistable of the same phase. Since $\mathcal{O}_X$ is $\sigma_3$-stable, the evaluation map ev$_E$ defined in \eqref{cone} is injective in the abelian category of semistable objects with the same phase as $\mathcal{O}_X$, hence the cokernel $K_E$ is $\sigma_3$-semistable. We claim that $K_E$ is $\sigma_3$-stable. Assume otherwise. Let $E_1$ be a $\sigma_3$-stable factor of $K_E$. Remark \ref{plane of having the same phase} implies that $v(E_1) = t_1v(E) +s_1v(\mathcal{O}_X)$. Similar to the argument in Lemma \ref{bound for h}, we deform the stability condition $\sigma_3$ towards the point $o'$, then
	\begin{equation*}
	0 \leq \lim\limits_{b \rightarrow 0}\text{Im}[Z_{(b,w)}(E_1)] = t_1 c(E) \leq \lim\limits_{b \rightarrow 0}\text{Im}[Z_{(b,w)}(E)] = c(E) =1.
	\end{equation*} 
	Since $t_1\text{c}(E) = t_1 \in \mathbb{Z}$, we have $t_1 = 0$ or $1$. Therefore, $\mathcal{O}_X$ is a subobject or a quotient of $K_E$. But Hom$_X(\mathcal{O}_X,K_E) = 0$ and since Hom$_X(E , \mathcal{O}_X) = 0$, we have Hom$_X(K_E,\mathcal{O}_X) =0$, a contradiction. \par
	Note that $K_E$ has Mukai vector $v(K_E) = v(E) - (r+s)v(\mathcal{O}_X) = (-s,H,-r)$ with the projection $pr(v(K_E)) =q$. Lemma \ref{q.3} shows that there is no wall for $K_E$ intersecting the open line segment $(\overline{oo'})$. Therefore, it is $\sigma_{(0,w)}$-stable where $w \gg 0$. By \cite[Lemma 6.18]{macri:intro-bridgeland-stability}, $H^0(K_E)$ is zero or a skyscraper sheaf and $H^{-1}(K_E)$ is a $\mu_H$-stable sheaf. If $H^0(K_E) \neq 0$, then for some $k>0$, we have
	\begin{equation*}
	v\big(H^{-1}(K_E)\big)^2 = (s,-H,r+k)^2 = -2sk < -2,
	\end{equation*}     
	a contradiction. Therefore $K_E=E'[1]$ for a $\mu_H$-stable coherent sheaf $E'$ on $X$. Since $E'^{\vee \vee}$ is also $\mu_H$-stable, we have 
	\begin{equation*}
	-2 \leq v(E'^{\vee\vee})^2 = v(E)^2 -2\text{rk}(E')l(E'^{\vee\vee}/E) = -2\text{rk}(E')l(E'^{\vee\vee}/E').
	\end{equation*}
	Since rk$(E') =s \geq 5$, we must have $E'= E'^{\vee\vee}$, so $E'$ is a locally free sheaf.
	
	To prove the final claim of part $(c)$, we find a stability condition such that $E'$ and $E(-H)[1]$ are stable of the same phase. Lemma \ref{phase one} implies that $E'[1]$ is $\sigma_4 \coloneqq \sigma_{(b_4,w_4)}$-stable where $b_4=-1/s$ and $w_4$ is arbitrary. We claim that $E'$ is stable with respect to the stability condition at the point $e'$. Let $t'$ be the point that the line segment $\overline{qp_u}$ intersects the line given by the equation $y = x(-s+1)$. Then the $x$-coordinate of the point $t'$ is equal to $-1/(2r-1)$ which is bigger than $-1/(r+1)$ if $r > 2$ and $t'=t$ if $r=2$. We claim that for $r=2$ the point $t = \big(-1/3 , (s-1)/3\big)$ cannot be the projection of a root. Indeed, if there exists a root $\delta = (\tilde{r},\tilde{c}H,\tilde{s})$ with $pr(\delta)= t$, then
	\begin{equation*}
	\dfrac{\tilde{c}}{\tilde{s}} = \dfrac{-1}{3} \;\;\;\; \text{and} \;\;\;\; \dfrac{\tilde{r}}{\tilde{s}} = \dfrac{s-1}{3}.
	\end{equation*}
	This implies $\abs{\tilde{s}} \geq 3$. Since $\delta^2=-2$, we have $\tilde{s}^2(s-3) = 9$ which is impossible for $s \geq 5$. Therefore there is a stability condition corresponding to any point on the line segment $(\overline{ot'}]$. By Lemma \ref{q.3}, there is no wall for $E'$ intersecting the line segment $(\overline{ot'}]$. Thus $\sigma_4$-stability of $E'$ implies that it is stable with respect to the stability condition at the point $e'$ and it has the same phase as $E(-H)[1]$. On the other hand, as we have seen there is no wall for $E(-H)[1]$ intersecting the line segment $(\overline{oe'}]$. Thus $E(-H)[1]$ is also stable with respect to the stability condition at the point $e'$, so there is no non-trivial homomorphism between $E'$ and $E(-H)[1]$. This finishes the proof of $(c)$.
\end{proof}
\subsection{The first wall} Let $M_C(r,2rs)$ be the moduli space of slope semistable rank $r$-vector bundles on the curve $C$ of degree $2rs$. Note that the vector bundle $F$ in $M_C(r,2rs)$ can be strictly semistable. The push-forward of the vector bundle $F$ to the surface $X$ has Mukai vector $v(i_*F) = (0,rH,2rs-r^2s)$. Lemma \ref{large vol} implies that $i_*F$ is semistable with respect to the stability conditions $\sigma_{(b,w)}$ in the Gieseker chamber which means $w$ is large enough and $b$ is arbitrary. By Proposition \ref{line wall}, part $(d)$, any wall for $i_*F$ is part of a line which goes through the point $p \coloneqq pr\big(v(i_*F)\big)$ if $r >2$ or it is a horizontal line segment if $r= 2$. The next proposition describes the location of the wall that bounds the Gieseker chamber for $i_*F$. Recall that $p_v=pr(v)$ and $p_u=pr(u)$ where $v =(r,H,s)$ and $u=v(i_*F)-v$. 
\begin{Prop}\label{first wall.1}
	Given a vector bundle $F \in M_C(r,2rs)$, the wall that bounds the Gieseker chamber for $i_*F$ is not below the line segment $\overline{p_up_v}$ and it coincides with the line segment $\overline{p_up_v}$ if and only if $F$ is the restriction of a vector bundle $E \in M_{X,H}(v)$ to the curve $C$. 
\end{Prop}
\begin{proof}
	Assume that the wall $\mathcal{W}_{i_*F}$ that bounds the Gieseker chamber for $i_*F$, is below or on the line segment $\overline{p_up_v}$, see Figure \ref{wall.3}.
	\begin{figure}[h]
		\begin{centering}
			\definecolor{zzttqq}{rgb}{0.27,0.27,0.27}
			\definecolor{qqqqff}{rgb}{0.33,0.33,0.33}
			\definecolor{uququq}{rgb}{0.25,0.25,0.25}
			\definecolor{xdxdff}{rgb}{0.66,0.66,0.66}
			
			\begin{tikzpicture}[line cap=round,line join=round,>=triangle 45,x=1.0cm,y=1.0cm]
			
			\filldraw[fill=gray!40!white, draw=white] (0,0)--(-.85,.73)--(-.85, 2.6)--(1.5,2.6)--(1.5,2.2)--(0,0);

			\draw[->,color=black] (-2.3,0) -- (3,0);

			
			

			
			\draw [] (0,0) parabola (1.61,2.5921); 
			\draw [] (0,0) parabola (-1.61,2.5921);

			\draw [ color=black] (-.85,.71)--(1.5,2.2); 
			\draw [ color=black, dashed] (-2,0)--(-.85,.71);

			\draw [dashed, color=black] (0,0)--(-.85,.73);
			\draw [dashed, color=black] (-.85, 2.6)--(-.85,.73);
			\draw [dashed, color=black] (0,0)--(1.5,2.2);
			\draw [dashed, color=black] (1.5,2.6)--(1.5,2.2);
			
			\draw [color=black, dashed] (.6,.86) -- (1,1);
			\draw [color=black] (-.54,.48) -- (.6,.86);
			\draw [color=black, dashed] (-2,0) -- (-.54,.48);
			

			\draw[->,color=black] (0,0) -- (0,2.8);
			
			\draw (-2,0) node [below] {$p$};
			\draw (-.82,.78) node [left] {$p_u$};
			\draw (1.45,2.2) node [right] {$p_v$};
			\draw (1,1) node [right] {$\mathcal{W}_{i_*F}$};
			\draw (0,0) node [below] {$o$};
			\draw (.55,.86) node [above] {$q_2$};
			\draw (-.4,.43) node [above] {$q_1$};
			\draw (.2,1.3) node [above] {$\tilde{o}$};
			
			\begin{scriptsize}
			
			\fill [color=black] (0,1.25) circle (1.1pt);
			\fill [color=black] (-2,0) circle (1.1pt);
			\fill [color=black] (1.5,2.2) circle (1.1pt);
			\fill [color=black] (-.85,.71) circle (1.1pt);
			\fill [color=black] (-.54,.48) circle (1.1pt);
			\fill [color=black] (.6,.86) circle (1.1pt);
			\fill [color=black] (0,0) circle (1.1pt);
			
			\end{scriptsize}
			
			\end{tikzpicture}
			
			\caption{The first wall $\mathcal{W}_{i_*F}$}
			
			\label{wall.3}
			
		\end{centering}
		
	\end{figure}
	
	Suppose the stability condition $\sigma_{(0,w')}$ is on the wall $\mathcal{W}_{i_*F}$. Then there is a destabilising sequence $ F_1 \hookrightarrow i_*F \twoheadrightarrow F_2$ 
	of objects in $\mathcal{A}(0)$ such that $F_1$ and $F_2$ are $\sigma_{(0,w')}$-semistable of the same phase as $i_*F$ and
	\begin{equation}\label{assumption for phase}
	\phi_{(0,w)}(F_1) > \phi_{(0,w)}(i_*F) \;\;\; \text{for} \;\; w < w'.
	\end{equation}
	Taking cohomology gives a long exact sequence of sheaves
	\begin{equation}\label{exact.12}
	0 \rightarrow H^{-1}(F_1) \rightarrow 0 \rightarrow H^{-1}(F_2) \rightarrow H^0(F_1) \xrightarrow{d_0} i_*F \xrightarrow{d_1} H^0(F_2) \rightarrow 0.
	\end{equation}
	Thus $H^{-1}(F_1) = 0$ and $H^0(F_1) \cong F_1$. Let $v(F_1) = \big(r',c'H,s'\big)$. If $r' =0$, then the projection $pr(v(F_1))$ lies on the $x$-axis. By Remark \ref{plane of having the same phase}, $F_1$ and $i_*F$ cannot have the same phase with respect to $\sigma_{(0,w')}$ unless $pr(v(F_1)) = pr(v(i_*F))$, i.e. $v(F_1) = k v(i_*F)$ for some $k \in \mathbb{R}$, which is in contradiction to \eqref{assumption for phase}. Hence $r' >0$.\par
	Let $T(F_1)$ be the maximal torsion subsheaf of $F_1$ and $F_1/T(F_1)$ be its torsion-free part. Let $v\big(T(F_1)\big) = (0,\tilde{r}H,\tilde{s})$. Right-exactness of the underived pull-back $i^*$ applied to the short exact sequence $T(F_1) \hookrightarrow F_1 \twoheadrightarrow F_1/T(F_1)$ implies that   
	\begin{equation}\label{bound for rank}
	\text{rank}(i^*F_1) \leq \text{rank}\big(i^*T(F_1)\big) + \text{rank}\big(i^*\big(F_1/T(F_1)\big)\big). 
	\end{equation}
	Since $F$ is a vector bundle on $C$ and $T(F_1)$ is a subsheaf of $i_*F$, we have rank$(i^*T(F_1)) = \tilde{r}$. Thus inequality \eqref{bound for rank} implies that $\text{rank}(i^*F_1) \leq \tilde{r} + r'$. Let $v\big(H^0(F_2)\big) = \big(0,c''H,s''\big)$. Then the right-exactness of $i^*$ implies
	\begin{equation}\label{above}
	\text{rank}\big(i^*\ker d_1 \big) \leq \text{rank}(i^*F_1) \;\;\; \Rightarrow \;\;\; r-c'' \leq r'+\tilde{r}.
	\end{equation} 
	Let $q_1$ and $q_2$ be the intersection points of the wall $\mathcal{W}_{i_*F}$ with the line segments $\overline{op_u}$ and $\overline{op_v}$, respectively. Then Lemma \ref{bound for slope} implies that  
	\begin{equation*}
	\dfrac{1}{r} \leq \mu_H^{-}(F_1) = \mu_H^{-}\big(F_1/T(F_1)\big) \;\;\;\; \text{and} \;\;\;\; \mu_H^{+}\big(H^{-1}(F_2)\big) \leq \dfrac{1-r}{r}.
	\end{equation*}
	Therefore, 
	\begin{align*}
	\dfrac{r-c''-\tilde{r}}{r'} = \mu_H\big(F_1/T(F_1)\big) -  \mu_H\big(H^{-1}(F_2)\big)\; \geq \;\;\;\;\;\;\;\;\;\;\;\;\;\;\;\;\;\;\;\;\;\;\;\;\\
	\;\;\;\;\;\;\;\;\;\;\;\mu_H^{-}\big(F_1/T(F_1)\big) -  \mu_H^{+}\big(H^{-1}(F_2)\big) \geq \dfrac{1}{r} - \dfrac{1-r}{r} = 1.
	\end{align*}
	Combined with the inequalities \eqref{above}, this is only possible if all these inequalities are equalities, i.e. $r' =r-c''-\tilde{r}$,
	\begin{equation}\label{equation first}
	\mu_H^{+}\big(H^{-1}(F_2)\big)  = \mu_H\big(H^{-1}(F_2)\big) = \dfrac{1-r}{r},
	\end{equation}
	and
	\begin{equation}\label{equation second}
	\dfrac{c'-\tilde{r}}{r-c''-\tilde{r}} = \mu_H\big(F_1/T(F_1)\big) = \mu_H^{-}\big(F_1/T(F_1)\big) =  \dfrac{1}{r}. 
	\end{equation}
	Since $c'' \geq 0$ and $\tilde{r} \geq 0$, the denominators of the first and last sentences in \eqref{equation second} imply that $c'' = \tilde{r} = 0$, hence $c' =1$, $v(F_1) = (r,H,s')$ and $v(T(F_1)) = (0,0,\tilde{s})$. However $T(F_1)$ cannot be a skyscraper sheaf because $T(F_1)$ is a subsheaf of $i_*F$. Thus $T(F_1) = 0$ and $F_1$ is torsion-free. Moreover, the equations \eqref{equation first} and \eqref{equation second} imply that $F_1$ and  $H^{-1}(F_2)$ are $\mu_H$-semistable sheaves. These sheaves are indeed $\mu_H$-stable because their rank and degree $\frac{\text{ch}_1(-).H}{H^2}$ are co-prime.  
	
	The point $pr\big(v(F_1)\big)$ lies on the line through the points $q_1$ and $q_2$. Moreover, $\frac{\text{rk}(F_1)}{\text{c}(F_1)} = r$, so it also lies on the line through the origin and $p_v$. Thus $pr\big(v(F_1)\big)= q_2$, see Figure \ref{wall.3}.
	On the other hand, $\mu_H$-stability of $F_1$ implies that $v(F_1)^2 = 2r(s-s') \geq -2$, i.e. $ s' \leq s$. Hence $q_2$ cannot be on the open line segment $(\overline{op_v})$. Therefore, the wall $\mathcal{W}_{i_*F}$ is above or on the line segment $\overline{p_up_v}$.   
	
	If the wall $\mathcal{W}_{i_*F}$ coincides with the line segment $\overline{p_up_v}$, then $v(F_1) = \big(r,H,s\big)$. The non-zero morphism $d_0$ in the long exact sequence \eqref{exact.12} factors via the morphism $d_0' \colon i_*F_1|_C \rightarrow i_*F$. The objects $i_*F_1|_C$ and $i_*F$ have the same Mukai vector and so have the same phase. Proposition \ref{prop.1} implies that $i_*F_1|_C$ is $H$-Gieseker stable. Hence the morphism $d_0'$ is injective. Since $i_*F_1|_C$ and $i_*F$ have the same Mukai vector, we must have isomorphism $F_1|_C \cong F$.  
\end{proof}
Now instead of checking the possible walls above the line segment $\overline{p_up_v}$, we consider the stability conditions of form $\sigma_{(0,w)}$ which are close to the point $(0,1)$ and examine the Harder-Narasimhan filtrations. Given a semistable vector bundle $F \in M_C(r,2rs)$, the $\sigma_{(0,w)}$-semistability of $i_*F$ for $w \gg 0$ implies that it does not have any subobject $F'$ in $\mathcal{A}(0)$ with ch$_1(F') = 0$ because $\phi_{(0,w)}(F') =1 > \phi_{(0,w)}(i_*F)$. Proposition \ref{polygon} for $i_*F$ implies that there exists $w^* >0$ such that for every stability condition $\sigma_{(0,w)}$ where $\sqrt{1/(rs)}< w <w^*$, the HN filtration of $i_*F$ is a fixed sequence $0 = \tilde{E}_{0} \subset \tilde{E}_{1} \subset .... \subset \tilde{E}_{n-1} \subset  \tilde{E}_n=i_*F$
with the semistable factors $E_i = \tilde{E}_i/\tilde{E}_{i-1}$ for $1 \leq i \leq n$. Recall that the stability function $\overline{Z} \colon K(X) \rightarrow \mathbb{C}$ is defined as $\overline{Z}(E)= Z_{\big(0,\sqrt{1/rs}\big)}(E) = \text{rk}(E)-\text{s}(E) \,+\, i\,\text{c}(E)$ and the polygon $P_{i_*F}$ has vertices $\{p_i\}_{i=0}^{i=n}$ where $p_i = \overline{Z}(\tilde{E}_i)$. Let $T$ be a triangle with the vertices $z_1 \coloneqq \overline{Z}(v) = r-s +i$, $z_2 \coloneqq \overline{Z}(i_*F) = r^2s-2rs + i\,r$ and the origin.
\begin{Lem}\label{inside poly}
	The polygon $P_{i_*F}$ for any vector bundle $F \in M_C(r,2rs)$ is contained in the triangle $T = \triangle oz_1z_2$ and they coincide if and only if the bundle $F$ is the restriction of a vector bundle $E \in M_{X,H}(v)$ to the curve $C$, see Figure \ref{wall.4}.
	\begin{figure} [h]
		\begin{centering}
			\definecolor{zzttqq}{rgb}{0.27,0.27,0.27}
			\definecolor{qqqqff}{rgb}{0.33,0.33,0.33}
			\definecolor{uququq}{rgb}{0.25,0.25,0.25}
			\definecolor{xdxdff}{rgb}{0.66,0.66,0.66}
			
			\begin{tikzpicture}[line cap=round,line join=round,>=triangle 45,x=1.0cm,y=1.0cm]

			\draw[->,color=black] (0,0) -- (0,3.5);
			\draw[->,color=black] (-3,0) -- (3,0);
			\draw[color=black] (0,0) -- (2.5,3.3);
			\draw[color=black] (0,0) -- (-1.5,1);
			\draw[color=black] (-1.5,1) -- (2.5,3.3);
			
			\draw[color=black] (0,0) -- (-.5,1);
			\draw[color=black] (.5,2) -- (2.5,3.3);
			\draw[color=black] (.5,2) -- (-.5,1);
			
			\draw (0,0) node [below] {$o$};
			\draw (3,0) node [right] {Re$[\overline{Z}(-)]$};
			\draw (0,3.5) node [above] {Im$[\overline{Z}(-)]$};
			\draw (-1.5,1) node [left] {$z_1$};
			\draw (2.6,3.3) node [above] {$z_2$};
			\draw (-.5,1) node [left] {$p_1$};
			\draw (0.4,2) node [above] {$p_2$};
			
			\begin{scriptsize}
			
			\fill [color=black] (0,0) circle (1.1pt);
			\fill [color=black] (2.5,3.3) circle (1.1pt);
			\fill [color=black] (-1.5,1) circle (1.1pt);
			\fill [color=black] (-.5,1) circle (1.1pt);
			\fill [color=black] (.5,2) circle (1.1pt);

			\end{scriptsize}
			
			\end{tikzpicture}
			
			\caption{The polygon $P_{i_*F}$ is inside the triangle $T$}
			
			\label{wall.4}
			
		\end{centering}
		
	\end{figure}
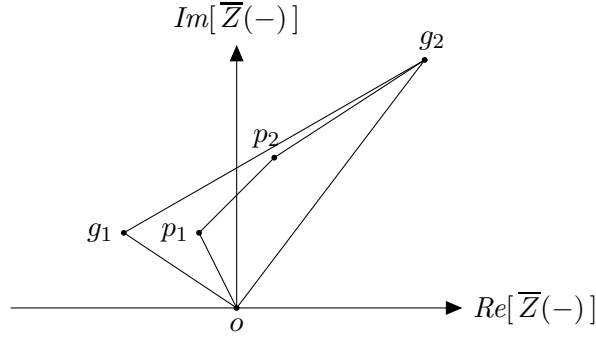	
\end{Lem} 
\begin{proof}
	The polygon $P_{i_*F}$ is convex. Thus for the first statement we only need to show the following two conditions are satisfied: firstly, the line through $\overline{op_1}$ is between the lines through $\overline{oz_1}$ and $\overline{oz_2}$, and secondly the line through $\overline{z_2p_{n-1}}$ is between the lines through $\overline{z_2z_1}$ and $\overline{z_2o}$. By definition, the points $\{p_i\}_{i=0}^{i=n}$ are on the left hand-side of the line segment $\overline{oz_2}$. Therefore it suffices to show that 
	\begin{equation}\label{the first condition}
	-\dfrac{\text{Re}[\overline{Z}(E_1)]}{\text{Im}[\overline{Z}(E_1)]} \leq -\dfrac{\text{Re}[z_1]}{\text{Im}[z_1]}  \;\;\;\;\;\; \text{and} \;\;\;\;   -\dfrac{\text{Re}[z_2-z_1]}{\text{Im}[z_2-z_1]} \leq -\dfrac{\text{Re}[\overline{Z}(E_n)]}{\text{Im}[\overline{Z}(E_n)]}. 
	\end{equation}   
	Let $v(E_1) = (r_1,c_1H,s_1)$. We have $0 < \text{Im}[Z_{(0,w)}(E_1)] = c_1 \leq \text{Im}[Z_{(0,w)}(i_*F)]= r$.	
	Assume for a contradiction that the first inequality in \eqref{the first condition} does not hold, then   
	\begin{equation}\label{k.1}
	\dfrac{s_1}{c_1} - \dfrac{r_1}{c_1} > s-r.
	\end{equation}
	Therefore, the point $q_1 \coloneqq (r_1/c_1 , s_1/c_1)$ is above the line $L_1$ given by the equation $y-x= s-r$. Proposition \ref{first wall.1} implies that $i_*F$ is $\sigma_{(0,\tilde{w})}$-semistable where $k(0,\tilde{w}) = \tilde{o}$ is on the line segment $\overline{p_up_v}$, see Figure \ref{wall.3}. Therefore, 
	\begin{equation}\label{k.2}
	\phi_{(0,\tilde{w})}\big(E_1\big) \leq \phi_{(0,\tilde{w})}(v)    \;\;\; \Rightarrow \;\;\; \dfrac{s_1}{c_1} - \dfrac{r_1}{c_1}(rs\tilde{w}^2) \leq s-r (rs\tilde{w}^2).
	\end{equation}
	This shows $q_1$ is below or on the line $L_2$ given by the equation $y-x(rs\tilde{w}^2) = s-r^2s\tilde{w}^2$, see Figure \ref{dashed.1}. Since the point of intersection of the lines $L_1$ and $L_2$ is $(r,s)$, we must have 
	\begin{equation*}
	r < \dfrac{r_1}{c_1} \;\;\; \Rightarrow \;\;\; r\, \leq\, \dfrac{r_1}{c_1} - \dfrac{1}{c_1} \leq \dfrac{r_1}{c_1} - \dfrac{1}{r}.
	\end{equation*} 
	Therefore, the point $q_1$ is in the dashed area in Figure \ref{dashed.1}. The point on the line $L_1$ with the first coordinate $r+1/r$, which is denoted by $q'$, has the second coordinate $s+1/r$. On the other hand, $E_1$ is $\sigma_{(0,w)}$-semistable where $\sqrt{1/(rs)} < w<w^*$, so Lemma \ref{new formula} implies that 
	\begin{equation*}
	-2c_1^2 \leq c_1^2(H^2) -2r_1s_1 = c_1^2(2rs) -2r_1s_1 \;\;\; \Rightarrow \;\;\;  \dfrac{r_1s_1}{c_1^2} \leq rs +1.
	\end{equation*}
	Therefore, the point $q_1$ is below or on the hyperbola with equation $xy= rs+1$. But the point $q'$ and so whole of the dashed area is above the hyperbola, a contradiction.
	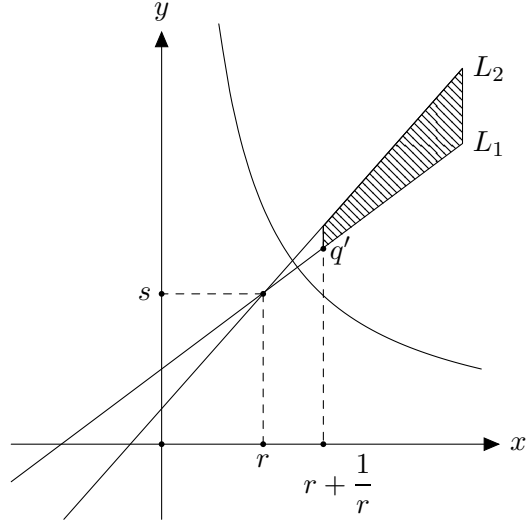
\begin{figure} [h]
		\begin{centering}
			\definecolor{zzttqq}{rgb}{0.27,0.27,0.27}
			\definecolor{qqqqff}{rgb}{0.33,0.33,0.33}
			\definecolor{uququq}{rgb}{0.25,0.25,0.25}
			\definecolor{xdxdff}{rgb}{0.66,0.66,0.66}
			
			\begin{tikzpicture}[line cap=round,line join=round,>=triangle 45,x=1.0cm,y=1.0cm]
			
			%
			\draw[pattern=north west lines, pattern color=black, thin] (1.9,2.43)--(1.9,2.6)--(4,5)--(4,4);
			
			
			\draw[scale=0.5,domain=1.26:8.5,smooth,variable=\x,black] plot ({\x},{14/\x});
			
			\draw[->,color=black] (0,-1) -- (0,5.5);
			\draw[->,color=black] (-2,0) -- (4.5,0);
			
			\draw[color=black] (-1.3,-1) -- (4,5);
			\draw[color=black] (-2,-.5) -- (4,4);
			\draw[color=black, dashed] (1.35,2) -- (1.35,0);
			\draw[color=black, dashed] (1.35,2) -- (0,2);
			\draw[color=black, dashed] (1.9, 2.43) -- (1.9,0);
			\draw[color=black] (1.9, 2.6) -- (1.9, 2.43);
			
			\draw (1.35,0) node [below] {$r$};
			\draw (0,2) node [left] {$s$};
			\draw (1.85, 2.4) node [right] {$q'$};
			\draw (2,0) node [below] {$r+\frac{1}{r}$};
			\draw (4,5) node [right] {$L_2$};
			\draw (4,4) node [right] {$L_1$};
			\draw (0,5.5) node [above] {$y$};
			\draw(4.5,0) node [right] {$x$};
			\begin{scriptsize}
			
			\fill [color=black] (0,0) circle (1.1pt);
			\fill [color=black] (1.35,2) circle (1.1pt);
			\fill [color=black] (0,2) circle (1.1pt);
			\fill [color=black] (1.35,0) circle (1.1pt);
			\fill [color=black] (1.9, 2.43) circle (1.1pt);
			\fill [color=black] (1.9,0) circle (1.1pt);
			
			\end{scriptsize}
			
			\end{tikzpicture}
			
			\caption{The point $q_1$ is in the dashed area}
			
			\label{dashed.1}
			
		\end{centering}
		
	\end{figure}
	
	Similarly, if the semistable factor $E_n$ with Mukai vector $v(E_n) = (r_n,c_nH,s_n)$ does not satisfy the second inequality in \eqref{the first condition}, then the point $q_n \coloneqq (r_n/c_n , s_n/c_n)$ is below the line $L_1'$ by the equation $y= x-s(r-1)+r/(r-1)$ and is above or on the line $L_2'$ with the equation $y = x(rs\tilde{w}^2) -s(r-1)+r^2s\tilde{w}^2/(r-1)$. Since the point of intersection of these two lines is $\big(-r/(r-1) , -s(r-1)\big)$, we have 
	\begin{equation*}
	\dfrac{r_n}{c_n} < \dfrac{-r}{r-1} \;\;\; \Rightarrow \;\;\; \dfrac{r_n}{c_n} \,\leq\, \dfrac{-r}{r-1} - \dfrac{1}{c_n(r-1)} \leq \dfrac{-r}{r-1} - \dfrac{1}{(r-1)^2}.
	\end{equation*}
	Then the same argument as above leads to a contradiction for $s \geq r$.\par
	If the vector bundle $F$ is the restriction of a vector bundle $E \in M_{X,H}(v)$, then Proposition \ref{prop.1} implies that the HN factors of $i_*E|_C$ with respect to the stability conditions close to the point $(0,1)$, are $E$ and $E(-H)[1]$. Therefore, the polygon $P_{i_*F}$ coincides with the triangle $T$. Conversely, assume for a vector bundle $F \in M_C(r,2rs)$, we have $P_{i_*F} = T$. Then $z_1 =p_1$ so $v(E_1) =(r+k,H,s+k)$ for some $k \in \mathbb{Z}$. The point $q_1 = (r+k,s+k)$ is on the line $L_1$. Proposition \ref{first wall.1} and inequality \eqref{k.2} imply that the point $q_1$ is below on on the line $L_2$, thus $k \geq 0$ (see Figure \ref{dashed.1}). Since c$(E_1) = 1$ is minimal, lemma \ref{q.3} implies that $E_1$ is $\sigma_{(0,w)}$-stable. Therefore $v(E_1)^2 = -2k(r+k+s) \geq -2$ which gives $k=0$ and Proposition \ref{first wall.1} implies that $F$ is the restriction of the vector bundle $E_1 \in M_{X,H}(v)$.   	
\end{proof}
\subsection{The maximum number of global sections}  
The next proposition shows that any vector bundle $F \in M_C(r,2rs)$ with high enough number of global sections is the restriction of a vector bundle on the surface.
\begin{Prop}\label{prop.2}
	Let $F$ be a slope-semistable rank $r$-vector bundle on the curve $C$ of degree $2rs$, where $r \geq 2$ and $s \geq \max \{5,r \}$. If $h^0(C,F) \geq r+s$, then $F$ is the restriction of a unique vector bundle $E \in M_{X,H}(v)$ to the curve $C$. In other words, the morphism $\psi \colon M_{X,H}(v) \rightarrow \mathcal{BN} = M_C(r,2rs,r+s)$, which sends a vector bundle to its restriction, is bijective. 
\end{Prop}
\begin{proof}
	If the vector bundle $F \in \mathcal{BN}$ is the restriction of a vector bundle $E \in M_{X,H}(v)$, then $E$ is a Harder-Narasimhan factor of $i_*F$ with respect to $\sigma_{(0,w)}$ where $ \sqrt{1/rs} < w < w^*$. Thus the uniqueness of the Harder-Narasimhan filtration implies that $\psi$ is injective. \par
	For the surjectivity part, by Lemma \ref{inside poly} we only need to show for any $F \in \mathcal{BN}$ the polygon $P_{i_*F}$ coincides with the triangle $T = \triangle \, oz_1z_2$. Assume for a contradiction that $P_{i_*F}$ is strictly inside $T$. Since the vertices of $P_{i_*F}$ are Gaussian integers, $P_{i_*F}$ must be contained in the polygon $oz_1'z_2'z_2$, where $z_1' = r-s+1 \,+ \,i$ and $\,z_2' = s(r-2)+r-r/(r-1) \,+ \, 2i $, see Figure \ref{wall.12}.
	\begin{figure}[h]
		\begin{centering}
			\definecolor{zzttqq}{rgb}{0.27,0.27,0.27}
			\definecolor{qqqqff}{rgb}{0.33,0.33,0.33}
			\definecolor{uququq}{rgb}{0.25,0.25,0.25}
			\definecolor{xdxdff}{rgb}{0.66,0.66,0.66}
			
			\begin{tikzpicture}[line cap=round,line join=round,>=triangle 45,x=1.0cm,y=1.0cm]

			\draw[->,color=black] (0,0) -- (0,3.7);
			\draw[->,color=black] (-2,0) -- (5.4,0);
			
			\draw[color=black] (0,0) -- (5,3.5);
			\draw[color=black] (0,0) -- (-1.2,1);
			\draw[color=black] (-1.2,1) -- (5,3.5);
			
			\draw[color=black] (0,0) -- (-.8,1);
			\draw[color=black] (1.3,2) -- (5,3.5);
			\draw[color=black] (1.3,2) -- (-.8,1);
			
			
			\draw (0,0) node [below] {$o$};
			\draw (5.4,0) node [right] {Re$[\,\overline{Z}(-)\,]$};
			\draw (0,3.6) node [above] {Im$[\,\overline{Z}(-)\,]$};
			\draw (-1.15,1) node [left] {$z_1$};
			\draw (-.7,.95) node [right] {$z_1'$};
			\draw (5,3.5) node [above] {$z_2$};
			\draw (1.3,2) node [above] {$z_2'$};

			\begin{scriptsize}
			
			\fill [color=black] (0,0) circle (1.1pt);
			\fill [color=black] (5,3.5) circle (1.1pt);
			\fill [color=black] (-1.2,1) circle (1.1pt);
			\fill [color=black] (-.8,1) circle (1.1pt);
			\fill [color=black] (1.3,2) circle (1.1pt);
			
			\end{scriptsize}
			
			\end{tikzpicture}
			
			\caption{The polygon $P_{i_*F}$ is inside the polygon $oz_1'z_2'z_2$}
			
			\label{wall.12}
			
		\end{centering}
		
	\end{figure}
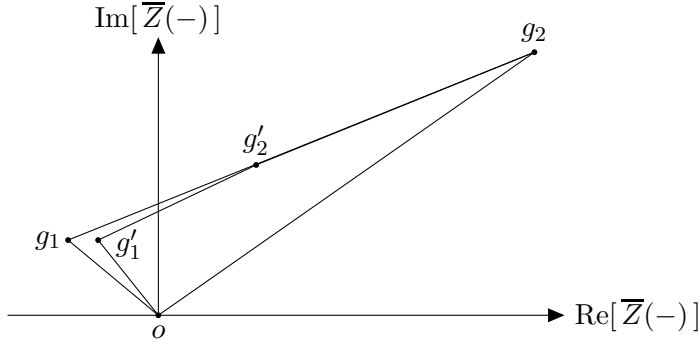
	
	The convexity of the polygon $P_{i_*F}$ and the polygon $oz_1'z_2'z_2$ gives $$\sum_{i=1}^{n} \lVert \overline{p_ip_{i-1}} \rVert \leq \lVert \overline{oz_1'} \rVert + \lVert \overline{z_1'z_2'} \rVert + \lVert \overline{z_2'z_2} \rVert \eqqcolon l_{\text{in}}.$$ 
	Note that $\lVert . \rVert$ is the non-standard norm defined in \eqref{definition of norm}. In our case, we have $H^2 =2rs$, so $\lVert x+ iy \rVert = \sqrt{ x^2 + (4rs +4) y^2}$. Let $l \coloneqq \lVert \overline{oz_1} \rVert + \lVert \overline{z_1z_2} \rVert$, then 
	\begin{equation*}
	l-l_{in} =  \lVert \overline{oz_1} \rVert - \lVert \overline{oz'_1} \rVert + \lVert \overline{z_1z_2'} \rVert - \lVert \overline{z'_1z'_2} \rVert.
	\end{equation*}
	Define $\epsilon \coloneqq \frac{2rs-r^2s}{2} + \frac{l}{2} - (r+s)$. Since $r+s \leq h^0(F)$, Proposition 3.4 implies
	\begin{equation*}
	\dfrac{2rs-r^2s}{2} + \dfrac{l}{2} - \epsilon\; = \; r+s\; \leq h^0(F) \leq \dfrac{\chi(F)}{2} + \dfrac{1}{2} \sum_{i=1}^{n}\lVert \overline{p_ip_{i-1}}\rVert  \leq \dfrac{2rs-r^2s}{2} + \dfrac{l_{\text{in}}}{2}.  
	\end{equation*} 
	Thus $l-l_{in} \leq 2\epsilon$. We have $l = \sqrt{(s+r)^2+4} + \sqrt{(s(r-1)^2 +r)^2 +4(r-1)^2}$, so
	\begin{equation*}
	2 \epsilon = \dfrac{4}{\sqrt{(r+s)^2 +4} + (r+s)} + \dfrac{4(r-1)^2}{\sqrt{ \big(s(r-1)^2+r \big)^2+4(r-1)^2 } + \big(s(r-1)^2+r \big) }.
	\end{equation*}  
	Therefore, we have $2 \epsilon \leq \frac{4}{2(r+s)} + \frac{4(r-1)^2}{2\big(s(r-1)^2+r \big)} \leq \frac{2}{r+s} + \frac{2}{s} \leq \frac{24}{35}$, because we assumed $r \geq 2$ and $s \geq \max\{r,5\}$.
	We will show $l-l_{in} \geq  0.6868$, thus $l-l_{in} > 2\epsilon$ which gives a contradiction.  
	
	We have $\lVert \overline{oz_1} \rVert - \lVert \overline{oz'_1} \rVert = \sqrt{4rs+4+(r-s)^2}- \sqrt{4rs+4+(r-s+1)^2}$, hence
	\begin{equation*}
	f_1(r,s) \coloneqq \dfrac{2s-2r -1}{2\sqrt{(r+s)^2 +4}}  \leq  \dfrac{2s-2r -1}{\sqrt{(r+s)^2 + 4} + \sqrt{(r+s)^2 +5 +2(r-s)}} = \lVert \overline{oz_1} \rVert - \lVert \overline{oz'_1} \rVert.
	\end{equation*}
	Also, $\lVert \overline{z_1z_2'} \rVert - \lVert \overline{z_1'z_2'} \rVert  = \sqrt{4rs+4 + \left(\frac{s(r-1)^2 - r}{r-1}\right)^2} - \sqrt{4rs+4 + \left(\frac{s(r-1)^2 - r}{r-1} -1\right)^2}$, thus 
	\begin{equation*}
	f_2(r,s) \coloneqq \dfrac{\frac{s(r-1)^2+r}{r-1} - \frac{2}{r-1} -\frac{5}{2}}{\sqrt{4 + \left(\frac{s(r-1)^2 + r}{r-1}\right)^2     }}= \dfrac{\frac{s(r-1)^2-r}{r-1} -\frac{1}{2}}{\sqrt{4rs+4 + \left(\frac{s(r-1)^2 - r}{r-1}\right)^2     }} \leq \lVert \overline{z_1z_2'} \rVert - \lVert \overline{z'_1z'_2} \rVert.
	\end{equation*}
	The function $f_1(r,s)$ is positive unless $s=r \geq 5$, and $f_2(r,s) = \sqrt{\frac{m}{4+m} } + \frac{-\frac{2}{r-1} -\frac{5}{2}}{\sqrt{4+m}}$ where $m= \left( \frac{s(r-1)^2 + r}{r-1}   \right)^2 $, so $f_2$ is an increasing function with respect to both $r$ and $s$. Consider the following three cases:
	\begin{itemize}
		\item If $r \geq 4$, then $f_1(r,s) \geq f_1(s,s) \geq f_1(5,5) \geq -0.05 $ and $f_2(r,s) \geq f_2(4,5) \geq 0.78$.
		\item If $r=3$, then $f_1(3,s) = \frac{s-3.5}{\sqrt{(s+3)^2+4}} \geq f_1(3,5) \geq 0.18$ and $f_2(3,s) \geq f_2(3,5) \geq 0.685$.
		\item If $r=2$, then $f_1(2,s) +f_2(2,s) = 2. \frac{s-2.5}{\sqrt{(s+2)^2+4}} \geq \frac{5}{\sqrt{53}} \geq 0.6868$.
	\end{itemize}
	Therefore, $l-l_{in} \geq 0.6868$ as claimed, which is a contradiction. Thus the polygon $P_{i_*F}$ coincides with the triangle $T$ and the morphism $\psi \colon M_{X,H}(v) \rightarrow \mathcal{BN}$ is surjective in case $\aaa$. 
\end{proof}
\begin{Cor}
	Let $F$ be a slope-semistable rank $r$-vector bundle on the curve $C$ of degree $2rs$ such that $r \geq 2$ and $s \geq \max\{5,r\}$. Then $h^0(F) \leq r+s$ and if $h^0(F) = r+s$, then $F$ is slope-stable.
\end{Cor}     
\begin{proof}
	Using the same notations as in the proof of Proposition \ref{prop.2}, we have 
	\begin{equation*}
	h^0(F) \leq \; \dfrac{2rs-r^2s}{2}  + \dfrac{1}{2} \sum_{i=1}^{n} \lVert \overline{p_ip_{i-1}} \rVert \; \leq  \dfrac{2rs-r^2s}{2} + \dfrac{l}{2} = (r+s)+ \epsilon \overset{(*)}{<} r+s+1,
	\end{equation*}
	where $(*)$ is the result of the above computation which shows $\epsilon \leq \frac{12}{35}$. If $h^0(F) = r+s$, then by Proposition \ref{prop.2}, $F$ is the restriction of a stable vector bundle on the surface, so it is stable.  
\end{proof}

\section{The final results}\label{section.6}
In this section we prove the main results.
\begin{proof}[Proof of Theorem \ref{main.2}]
	By Proposition \ref{prop.2}, the morphism $\psi \colon M_{X,H}(v) \rightarrow \mathcal{BN}$ is bijective. Therefore Proposition \ref{prop.1} implies that any vector bundle $F$ in the Brill-Noether locus $\mathcal{BN}= M_C(r,2rs,r+s)$ is slope-stable and $h^0(F) = r+s$. The moduli space $M_{X,H}(v)$ is a smooth projective variety of dimension 2. Hence we only need to show derivative of the restriction map $d\psi$ is surjective. The Zariski tangent space to the Brill-Noether locus $\mathcal{BN}$ at the point $[F]$ is the kernel of the map 
	\begin{equation*}
	k_1 \colon \text{Ext}^1(F,F) \rightarrow \text{Hom} \big( H^0(C,F) , H^1(C,F) \big),
	\end{equation*}
	where any $f \colon F \rightarrow F[1] \in \text{Ext}^1(F,F) = \text{Hom}_{C}(F,F[1])$ goes to 
	\begin{equation*}
	k_1(f) = H^0(f) \colon \text{Hom}_{C}(\mathcal{O}_C,F) \rightarrow \text{Hom}_{C}(\mathcal{O}_C,F[1]),
	\end{equation*}
	see \cite[Proposition 4.3]{bhosle:brill-noether-loci-on-nodal-curves} for details. Note that the proof in \cite{bhosle:brill-noether-loci-on-nodal-curves} is valid for any family of simple sheaves on a variety. In addition, for any vector bundle $E$ in the moduli space $M_{X,H}(v)$,
	\begin{equation*}
	T_{[E]}\big( M_{X,H}(v) \big) = \text{Hom}_X(E,E[1]). 
	\end{equation*}
	Let $i \colon C \hookrightarrow X$ be the closed embedding of the curve $C$ into the surface $X$, then $Ri_*(-) = i_*(-)$ and for a vector bundle $E$ on $X$, we have $Li^*(E) = i^*(E)$. The derivative of the restriction map 
	\begin{equation*}
	d\psi \colon  	T_{[E]}M_{X,H}(v)  \rightarrow T_{[E|_C]} \mathcal{BN},
	\end{equation*} 
	sends any $f \colon E \rightarrow E[1] \in \text{Hom}_X(E,E[1])$ to its restriction $i^*f \colon i^*E \rightarrow i^*E[1] \in \ker(k_1)$. \par
	Define $h\colon \text{id}_{\mathcal{D}(X)} \rightarrow Ri_*Li^*$ as the natural transformation for the pair of adjoint functors $Li^* \dashv Ri_*$. Given a vector bundle $E$ in the moduli space $M_{X,H}(v)$ and a morphism $\varphi \in \text{Hom}_X(E,E[1])$, we have the commutative diagram 
	\begin{equation}\label{diag}
	\xymatrix{
		E\ar[r]^-{h_E}\ar[d]_-{\varphi}&i_*i^*E\ar[d]^-{i_*i^*\varphi}\\
		E[1]\ar[r]^-{h_{E[1]}}&i_*i^*E[1].
	}
	\end{equation}	
	Therefore the following diagram is also commutative
	\begin{equation*}\label{stacks}
	\xymatrix{
		\text{Hom}_X\big(E,E[1]\big)\ar[r]^-{d\psi}\ar[dr]_-{k_2\,\coloneqq \,h_{E[1]}\,\circ \, (-)}&\text{Hom}_C\big(i^*E,i^*E[1]\big)\ar[d]_-{\sim}^-{i_*(-) \, \circ \,h_E \,\eqqcolon k_3}\ar[r]^-{k_1}&\text{Hom}\big(H^0(C,i^*E),H^1(C,i^*E)\big).\\
		&\text{Hom}_X\big(E,i_*i^*E[1]\big)}
	\end{equation*}
	Consider the following distinguished triangle   
	\begin{equation}\label{distingushed rinagle}
	E[1] \xrightarrow{h_{E[1]}} i_*i^*E[1]\xrightarrow{g} E(-H)[2].
	\end{equation}	
	Given a morphism $\xi \in \text{ker}(k_1) \subseteq \text{Hom}_C\big(i^*E,i^*E[1]\big)$, we first claim that the composition $g \circ i_*\xi \circ h_E = 0$ vanishes, so there exist morphisms $\xi'$ and $\xi''$ such that the following diagram commutes. 
	\begin{equation}\label{final diagram}
	\xymatrix{
		E\ar[r]^-{h_E}\ar@{-->}[d]_-{\exists \, \xi'}&i_*i^*E \ar[r]\ar[d]_-{i_*{\xi}}&E(-H)[1]\ar@{-->}[d]_-{\exists \,\xi''}\\
		E[1]\ar[r]^-{h_{E[1]}}&i_*i^*E[1]\ar[r]^-{g}&\;E(-H)[2].\		}
	\end{equation}
	Since $\xi \in \text{ker}(k_1)$, the composition 
	\begin{equation*}
	\xi \circ i^*\text{ev}_E \colon i^*\mathcal{O}_X^{h^0(E)} \rightarrow i^*E \rightarrow i^*E[1] \in \text{Hom}_C\big( i^*\mathcal{O}_X^{h^0(E)},  i^*E[1] \big)
	\end{equation*} 
	vanishes. Thus the adjunction $Li^* \dashv Ri_*$ gives $i_*(\xi \circ i^*\text{ev}_E) \circ h_{ \mathcal{O}_X^{h^0(E)} } = 0$. We know $i_*i^*\text{ev}_E \circ h_{ \mathcal{O}_X^{h^0(E)} } = h_E \circ \text{ev}_E$, hence the morphism $$i_* \xi \circ (h_E \circ \text{ev}_E) \colon \mathcal{O}_X^{h^0(E)} \rightarrow i_*i^*E \rightarrow i_*i^*E[1]$$ vanishes. Moreover, Propositions \ref{prop.1} implies that we have the short exact sequence 
	\begin{equation*}
	0 \rightarrow E' \rightarrow \mathcal{O}_X^{h^0(E)} \xrightarrow{\text{ev}_E} E \rightarrow 0
	\end{equation*}
	in Coh$(X)$. Since Hom$_X\big(E',E(-H)[1]\big) = 0$, applying the functor Hom$_X\big(-,E(-H)[2]\big)$ gives the exact sequence
	\begin{equation*}
	0 \rightarrow \text{Hom}_X\big(E,E(-H)[2]\big) \xrightarrow{\Phi} \text{Hom}_X\big(\mathcal{O}_X^{h^0(E)},E(-H)[2]\big) \rightarrow \text{Hom}_X\big(E',E(-H)[2]\big) \rightarrow 0.
	\end{equation*}
	Since $\Phi\,\big(g \circ i_*\xi \circ h_E\big) = g \circ (i_*\xi \circ h_E \circ \text{ev}_E) = 0$, the injectivity of $\Phi$ implies $g \circ i_*\xi \circ h_E = 0$ as claimed. Therefore there exists morphism $\xi ' \colon E \rightarrow E[1]$ in the diagram \eqref{final diagram} such that $i_*\xi \circ h_E = h_{E[1]} \circ \xi'$. The commutative diagram \eqref{diag} gives $h_{E[1]} \circ \xi' = i_*i^*\xi' \circ h_E$. Therefore $i_*\xi \circ h_E = i_*i^* \xi' \circ h_E \in \text{Hom}_X\big(E, i_*i^*E[1]\big)$ and the isomorphism $k_3$ implies $\xi = i^* \xi'$ which shows $d\psi$ is surjective. 

\end{proof}

\begin{proof}[Proof of Theorem \ref{main.1}]
	Let $(X,H)$ be a polarised K3 surface with Pic$(X) = \mathbb{Z}.H$, and let $C$ be any curve in the linear system $|H|$. The moduli space $N=M_{X,H}(v)$ is a smooth projective K3 surface. There exists a Brauer class $\alpha \in \text{Br}(N)$ and a universal $(1 \times \alpha)$-twisted sheaf $\tilde{\mathcal{E}}$ on $X \times N$.\par
	Theorem \ref{main.2} implies that the moduli space $N$ is isomorphic to the Brill-Noether locus $\mathcal{BN}$ and the restriction of the universal twisted sheaf $\tilde{\mathcal{E}}|_{C \times \mathcal{BN}}$ is a universal $(1 \times \alpha)$-twisted sheaf on $C \times \mathcal{BN}$, so $v' =  v\big(\tilde{\mathcal{E}}|_{p \times \mathcal{BN} }\big) $ for a point $p$ on the curve $C$.\par
	Let $H'$ be a generic polarisation on $N$. Then the moduli space $M_{N, H'}^{\alpha}(v')$ of $\alpha$-twisted semistable sheaves on $N$ with respect to $H'$, is isomorphic to the original K3 surface $X$ (see e.g. \cite[Theorem 2.7.1]{yoshioka:preserve-coherent-sheaves-fm-transform-surface-2}). Therefore, $M_{\mathcal{BN},H'}^{\alpha}(v') \cong X$ which completes the proof of Theorem \ref{main.1}.   
\end{proof}



\bibliography{mybib}
\bibliographystyle{halpha}

\end{document}